\documentclass[11pt]{article}
\usepackage[utf8]{inputenc}
\usepackage[T1]{fontenc}             
\usepackage{lmodern}
\usepackage[french,english]{babel}
\usepackage{amsmath}
\usepackage{amsthm}
\usepackage{amssymb}
\usepackage{mathrsfs}
\usepackage{geometry}
\usepackage{graphicx}
\usepackage{comment}
\usepackage{url}
\usepackage{titlesec}
\usepackage[all]{xy}
\usepackage{enumitem}
\usepackage{amsfonts,amsxtra}
\usepackage{stmaryrd} 
\usepackage{tikz}
\usepackage[colorlinks=true]{hyperref}

\newtheorem{df}{Definition}[section]
\newtheorem{rem}[df]{Remark}
\newtheorem{ex}[df]{Example}
\newtheorem{thm}[df]{Theorem}
\newtheorem{pp}[df]{Proposition} 
\newtheorem{lm}[df]{Lemma}
\newtheorem{cor}[df]{Corollary}

\def\so{\mathfrak{so}}

\def\sl{\mathfrak{sl}}
\def\gl{\mathfrak{gl}}

\def\sp{\mathfrak{sp}}

\def\gg{\mathfrak{g}}
\def\hh{\mathfrak{h}}
\def\ppp{\mathfrak{p}}
\def\kk{\mathfrak{k}}

\def\uu{\mathfrak{u}}

\def\ss{\mathfrak{s}}

\def\osp{\mathfrak{osp}}
\def\hht{\tilde{\hh}}

\newcommand \C[1]{{\mathcal #1}}

\newcommand\CH{{\C H}}

\newcommand\CW{{\C W}}

\newcommand\CP{{\C P}}

\newcommand\al{{\alpha}}

\newcommand \bC{{\mathbb C}}
\newcommand \bE{{\mathbb E}}

\newcommand \bH{{\mathbb H}}
\newcommand \bR{{\mathbb R}}
\newcommand \bN{{\mathbb N}}

 \newcommand \wti[1]{{\widetilde {#1}}}

\newcommand\End{\operatorname{End}}

\newcommand\tr{\operatorname{tr}}

\newcommand\ad{\operatorname{ad}}

\newcommand\triv{\mathsf{triv}}

\numberwithin{equation}{section}

\ifx\undefined\bysame
\newcommand{\bysame}{\leavevmode\hbox to3em{\hrulefill}\,}
\fi

\begin{document}

\title{Symplectic Dirac operators for Lie algebras and graded Hecke algebras}

\author{ Dan Ciubotaru       \\
          \texttt{dan.ciubotaru@maths.ox.ac.uk},
           \and Marcelo De Martino\\
        \texttt{marcelo.demartino@maths.ox.ac.uk}
   \and Philippe Meyer\\
          \texttt{philippe.meyer@maths.ox.ac.uk}
	}
\date{Mathematical Institute\\ University of
          Oxford\\ Oxford, OX2 6GG, UK }

\maketitle
\begin{abstract}
We define a pair of symplectic Dirac operators $(D^+,D^-)$ in an algebraic setting motivated by the analogy with the algebraic orthogonal Dirac operators in representation theory. We work in the settings of $\mathbb Z/2$-graded quadratic Lie algebras $\gg=\kk+\ppp$ and of graded affine Hecke algebras $\mathbb H$.
\end{abstract}

\footnotetext{This research was supported by the EPSRC grant EP/N033922/1 (2016).}

\section{Introduction} 
The idea of symplectic spin geometry and symplectic Dirac operators originated with the work of Kostant \cite{Ko}, and it was developed substantially by K. Halbermann and her collaborators (see for example the textbook \cite{HH}) who studied geometric symplectic Dirac operators for symplectic manifolds in analogy to the classical Riemannian Dirac operator. 

In this paper, we introduce certain symplectic Dirac operators in an algebraic setting with a view towards applications to representation theory. We are also motivated by the analogy with the algebraic orthogonal Dirac operators in representation theory (e.g., \cite{Pa}, \cite{HP}, \cite{BCT}). We define pairs for symplectic Dirac elements $(D^+,D^-)$ in the setting of $\mathbb Z/2$-graded quadratic Lie algebras $\gg=\kk+\ppp$ and for graded Hecke algebras $\mathbb H$. In the Lie algebra case, the symplectic Dirac elements live in $U(\gg)\otimes \CW(\ppp+\ppp^*,\omega)$, where $U(\gg)$ is the universal enveloping algebra of $\gg$, while $\CW=\CW(\ppp+\ppp^*,\omega)$ is the Weyl algebra defined with respect to the symplectic form $\omega$ on $\ppp+\ppp^*$, $\omega(\alpha,Z)=\alpha(Z)$, for $\alpha\in\ppp^*$, $Z\in\ppp$.  The graded Hecke algebra $\mathbb H$ is a deformation of $\mathbb C[W]\# S(V)$, where $W$ is a finite Weyl group in the orthogonal group of its reflection representation $V$ with respect to a nondegenerate symmetric form $B$. The symplectic Dirac operators $D^+,D^-$ live in $\mathbb H\otimes \CW(V+V^*,\omega)$, where $\omega$ is again the natural symplectic form on $V+V^*$.

We compute formulas for the commutator $[D^+,D^-]$ (Theorems \ref{t:comm} and \ref{p:comm}), which could be viewed as analogues of Partasarathy's formula for the square of the orthogonal Dirac operator. For Lie algebras, this commutator is expressed in terms of the Casimir element $\Omega(\gg)$ of $\gg$ and the Casimir element $\Omega(\kk)$ of $\kk$, but unlike the case of the square of the orthogonal Dirac operator, here, $\Omega(\kk)$ occurs in all three possible ways: in $U(\gg)\otimes 1\subset U(\gg)\otimes \CW$, diagonally in $U(\gg)\otimes \CW$, and also in $1\otimes \CW$, via the map $\nu: \kk\to \mathfrak{so}(\ppp)\hookrightarrow \CW$. The image $\nu(\Omega(\kk))$ is not a scalar in $\CW$. The Weyl algebra has a canonical linear isomorphism $\eta$ onto $S(\ppp+\ppp^*)$ and by Kostant's result \cite[Proposition 2.6]{Ko3}, $\eta(\nu(\Omega(\kk)))$ has only degree $4$ and degree $0$ parts. We study $ \eta(\nu(\Omega(\kk)))$ and, in particular, show that the degree $4$ part is nonzero, while the degree $0$ part admits an interesting formula, Corollary \ref{c:k-0}, related to the strange Freudenthal-de Vries formula. 

In the case of the Hecke algebra $\mathbb H$, the analogous formula for $[D^+,D^-]$ leads to two interesting central elements $\Omega_W$ and $\Omega'_W$ in the group algebra of $W$. The element $\Omega_W$ appeared also in the theory of the orthogonal Dirac operator for graded Hecke algebras \cite{BCT}. In this paper, we explicitly compute it and its action on irreducible representations for the classical root systems, see section \ref{s:4.3}. The other element, \[\Omega'_W=\frac {1}{16}\sum_{\al,\beta>0,~s_\al(\beta)<0}k_\al k_\beta B(\al^\vee,\beta^\vee)^{-1} [s_\al, s_\beta]^2,\] is new. Here $s_\al$ is the reflection corresponding to the positive root $\al$  and $k_\al$ are the parameters of $\bH$.  In the formula for $[D^+,D^-]$, we are led naturally to consider $\nu'(\Omega'_W)$, the image in $\CW(V+V^*,\omega)$ under the embedding $\mathfrak{so}(V,B)\subset \CW$, where we regard $[s_\al,s_\beta]$ as an element of $\mathfrak{so}(V,B)$. (In fact, $\{[s_\al,s_\beta]:\al,\beta>0\}$ spans $\mathfrak{so}(V,B)$, see Proposition \ref{p:span}.) We show that $\nu'(\Omega'_W)$ is an $O(V,B)$-invariant element of $\CW$ (Proposition \ref{p:O-inv}), and therefore, by the well-known theory of dual pairs, recalled in section \ref{s:1}, it must equal a constant plus a multiple of the image of the Casimir element of $\mathfrak{sl}(2)$. In section \ref{s:4.4}, we determine precisely the relation between $\nu'(\Omega'_W)$ and $\Omega(\mathfrak{sl}(2))$ inside $\CW$.

As an application, we look at the actions of $D^+,D^-$ on representations, particularly unitary representations (in the settings of admissible $(\gg,K)$-modules and of finite dimensional $\mathbb H$-modules) and find certain generalisations of the classical Casimir inequality. 

One would expect that the constructions in this paper could be formalised in the setting of a cocommutative Hopf algebra acting on a symplectic module,  similarly to the general setting for an algebraic theory of the orthogonal Dirac operator from \cite{Fl}.

\section{Preliminaries: Weyl algebra}\label{s:1}
Let $k$ be a field of characteristic $0$ and let $(\C V,\omega)$ be a finite-dimensional symplectic vector space. Consider the Weyl algebra $\CW(\C V,\omega)$ defined by $\CW(\C V,\omega)=T(\C V)/I$
where $I$ is the two-sided ideal of the tensor algebra $T(\C V)$ generated by the elements of the form
$$v\otimes w-w\otimes v-\omega(v,w)\cdot1\qquad \forall v,w\in \C V.$$
Define $\epsilon:\C V\rightarrow {\rm End}(S(\C V))$ by $\epsilon(v)(P)=v\cdot P$ and define $i:\C V\rightarrow {\rm End}(S(\C V))$ to be unique derivation of degree $-1$ such that $i(v)(w)=\omega(v,w)$.

The linear map $\gamma:\C V\rightarrow {\rm End}(S(\C V))$ given by $\gamma(v)=\epsilon(v)+\frac{1}{2}i(v)$ extends to a homomorphism of associative algebras $\gamma:\CW(\C V,\omega)\rightarrow {\rm End}(S(\C V))$ that yields a linear isomorphism $\eta: \CW(\C V,\omega)\rightarrow S(\C V)$ given by 
\begin{equation}\label{e:eta}
\eta(x) = \gamma(x)(1),\qquad\forall x\in\CW.
\end{equation}
For each $x\in \CW$, we shall write $(x)_d$ for the degree $d$ component of 
$\eta(x) \in S(\C V) = \oplus_{d\in \bN} S^d(\C V)$. 
The inverse map of $\eta$ is the quantization map $Q$ that satisfies
$Q(v_1\cdot\ldots\cdot v_n)=\tfrac{1}{n!}\sum_{\sigma\in S_n}v_{\sigma(1)}\ldots v_{\sigma(n)}.$
Occasionally, we shall denote the product in the symmetric algebra by $x\cdot y$ while the product in the Weyl algebra product will be denoted by juxtaposition. Furthermore, the map $\mu : S^2(\C V)\rightarrow \sp(\C V,\omega)$ defined by
\begin{equation}\label{e:moment}
\mu(u\cdot v)(w)=\omega(u,w)v+\omega(v,w)u \qquad \forall u,v,w\in \C V
\end{equation}
is an isomorphism of Lie algebras, where the bracket on $S^2(\C V)$ is the Weyl commutator. The inverse map of $\mu$ satisfies
$$\mu^{-1}(f)=\frac{1}{2}\sum\limits_i f(v^i)\cdot v_i$$
where $\lbrace v_i\rbrace$ is a basis of $\C V$ and $\lbrace v^i\rbrace$ is the basis dual to the basis $\lbrace v_i\rbrace$ is the sense that $\omega(v_i,v^j)=\delta_{ij}$.  

If $V,V'$ are maximal isotropic subspaces of $\C V$ such that $\C V=V\oplus V'$, define $m:\C W(\C V,\omega)\rightarrow {\rm End}(S(V))$ by
\begin{equation}\label{e:polyaction}
m(v)(P)=v\cdot P, \qquad m(\alpha)(P)=i(\alpha)(P) \qquad \forall v\in V, ~ \forall \alpha\in V', ~ \forall P\in S(V).
\end{equation}

Suppose now that $k$ is $\bR$ or $\bC$, that $V=k^n$ is endowed with a nondegenerate (positive definite when $k=\bR$) symmetric bilinear form $B$ and that $\C V=V\oplus V^*$.  

For calculations, it may be convenient to work with coordinates. Let $\{v_i\}$ be an orthonormal basis of $(V,B)$. To distinguish between the elements in $V$ and those in $\C W(V\oplus V^*)$, it is sometimes convenient to denote by $e_i$ the image of $v_i$ in $\CW$. Let $\{v_i^*\}$ be the dual basis in $V^*$ and denote the image of $v_i^*$ in $\CW$ by $f_i$. In $\CW(V\oplus V^*)$, the convention is $[f_j,e_i]=\delta_{ij}$. 

Introduce the linear isomorphism 
\[\iota: V\oplus V^*\to V\oplus V^*, \quad \iota(v_i)=v_i^*,\ \iota(v_i^*)=v_i.
\]
The symplectic Lie algebra $\sp(V\oplus V^*,\omega)$ is isomorphic to $\mathfrak{sp}(2n,k)$ where $\mathfrak{sp}(2n,k)$ is the subalgebra of matrices $X\in \mathfrak{gl}(2n,k)$ such that $X^tJ+JX=0$, where $J=\left(\begin{matrix} 0& -I_n\\ I_n&0\end{matrix}\right)$. Hence $X=\left(\begin{matrix} A &B\\C &-A^t\end{matrix}\right),$ where $A,B,C$ are $n\times n$ matrices satisfying $B=B^t$ and $C=C^t$. Denote by $E_{ij}$ the $(i,j)$-elementary matrix. Under the isomorphisms $\mu:S^2(V\oplus V^*)\rightarrow\sp(V\oplus V^*,\omega)$ and $Q:S(V\oplus V^*)\rightarrow \CW(V\oplus V^*,\omega)$, the canonical basis of $\mathfrak{sp}(2n,k)$ corresponds to
\begin{align*}
e_{ij}-e_{n+j,n+i}&\mapsto E_{ij}+\frac 12\delta_{ij}\\
e_{i,j+n}+e_{j,i+n}&\mapsto M_{ij}\\
-e_{i+n,j}-e_{j+n,i}&\mapsto \Delta_{ij} ,\quad 1\le i,j\le n
\end{align*}
where
\begin{equation}
\begin{aligned}
\Delta_{ij}=f_if_j,\quad M_{ij}=e_ie_j,\quad E_{ij}=e_i f_j.
\end{aligned}
\end{equation}
Define $\ss:=Q\circ\mu^{-1}(\sp(V\oplus V^*,\omega))$. The Lie algebra $\mathfrak{gl}(n,k)$ is embedded diagonally into $\mathfrak{sp}(2n,k)$ by $A\mapsto \left(\begin{matrix} A &0\\0 &-A^t\end{matrix}\right)$ and so define $\mathfrak l:=Q\circ\mu^{-1}(\gl(n,k))$.
\vspace{0.2cm}

The map $\iota$ induces the transpose on $\mathfrak l$:
\begin{equation}
\iota: \mathfrak l\to \mathfrak l,\quad \iota(E_{ij})=E_{ji}.
\end{equation}
Let $\mathfrak k\subset \mathfrak l$ be the $(-1)$-eigenspace of $\iota$. Then $\mathfrak k$ is isomorphic to $\mathfrak{so}(n,k)$.

Define
\begin{equation}\label{sl2}
\begin{aligned}
\Delta&=\sum_{j=1}^n \Delta_{jj}, &\quad& X&&=-\frac 12\Delta,\\
E&=\sum_{i=1}^n E_{ii},&\quad& H&&=-E-\frac n2\\
r^2&=\sum_{j=1}^n M_{jj},&\quad& Y&&=\frac 12 r^2.
\end{aligned}
\end{equation}
It is easy to check (see \cite[(5.86),Theorem 5.6.9]{GW}) that $\lbrace X,H,Y\rbrace$ is a Lie triple and that $\mathfrak t=\text{span}\{X,H,Y\}\cong \mathfrak{sl}(2,k)$ commutes with $\mathfrak k\cong\mathfrak{so}(n,k)$ in $\C W$. 

Let $\C P$ denote the space of polynomials on $V^*$, equivalently $\C P=S(V)$. As before, this is a module for $\C W$ via the action $m:\C W\to\End(\C P)$ of (\ref{e:polyaction}) where $m(e)$ is the multiplication by $e\in V$ operator and $m(f)$ is the directional derivative for $f\in V^*$. Let $\C P^\ell$ denote the subspace of homogeneous polynomials of degree $\ell\ge 0$.

\begin{thm}[{\cite[Theorem 5.6.11]{GW}}] As an $\mathfrak{so}(n,k)\times \mathfrak {sl}(2,k)$-module, $\C P$ decomposes
\begin{equation}
\C P=\bigoplus_{\ell\ge 0} \CH^\ell(V)\otimes M_{-(\ell+\frac n2)},
\end{equation}
where $\CH^\ell(V)=\{p\in \C P^\ell\mid \Delta(p)=0\}$ is the space of spherical harmonic polynomials of degree $\ell$ and $M_{-(\ell+\frac n2)}$ is the Verma module of $\mathfrak t$ with highest weight $-(\ell+\frac n2)$.
\end{thm}

It is well-known that the Casimir element $\Omega(\mathfrak{sl}(2)) = H^2 + 2(XY + YX)\in \C W$ satisfies
$\eta( \Omega(\mathfrak{sl}(2)))\in S^4(V\oplus V^*)\oplus S^0(V\oplus V^*)$, where $\eta$ is the linear isomorphism $\C W\to S(V\oplus V^*)$ of (\ref{e:eta}). For convenience, we recall the values of its components.

\begin{pp}\label{p:sl2Cas}
We have $\eta(\Omega(\mathfrak{sl}(2))) = (\eta(H)^2 -\eta(\Delta)\eta(r^2))-3n/4$.
\end{pp}

\begin{proof}
From \cite[Proposition 2.6]{Ko2}, we know that $\eta(H^2)\in S^4(V\oplus V^*)\oplus S^0(V\oplus V^*)$. A straightforward computation gives $\eta(H^2)=\gamma(H^2)(1) = \eta(H)^2 - \tfrac{n}{4}$. On the other hand, one computes
\begin{align*}
4\eta(XY)=4\gamma(X)\gamma(Y)(1) &= 2\gamma(X)(\eta(r^2))\\
&= -\sum_j \gamma(f_j)(f_j\cdot \eta(r^2) + e_j) \\
&= -\eta(\Delta)\eta(r^2) - 2\sum_j f_j\cdot e_j - \tfrac{n}{2},
\end{align*}
as $i(f_j)(e_k)= \delta_{jk}$. Since $\Omega(\mathfrak{sl}(2)) = H^2 - 2H + 4XY$ and $\eta(H)=-\sum_j f_j\cdot e_j$, the claim follows.
\end{proof}

\section{Lie algebras} 

\subsection{Symplectic Dirac elements} Let $k$ be a field of characteristic $0$.
Let $(\gg,B)$ be a finite-dimensional quadratic Lie algebra and suppose that we have a $\mathbb{Z}_2$-gradation $\gg=\kk\oplus \ppp$ such that $\kk$ and $\ppp$ are $B$-orthogonal. Let $\lbrace w_k\rbrace$ be a basis of $\kk$, let $\lbrace w^k\rbrace$ be its dual basis, let $\lbrace z_i \rbrace$ be a basis of $\ppp$ and let $\lbrace z^i\rbrace$ be its dual basis.
\vspace{0.2cm}

We have
${\rm End}(\ppp)\cong \ppp\otimes\ppp^*$ and the identity corresponds to
$$D^-=\sum\limits_i z_i\otimes z_i^*.$$
Using $B$ we have ${\rm End}(\ppp)\cong \ppp\otimes\ppp$ and the identity corresponds to
$$D^+=\sum\limits_i z_i\otimes z^i.$$
In particular, $D^+$ and $D^-$ are independent of the choice of the basis $\lbrace z_i \rbrace$. Let $\omega$ be the symplectic form on $\ppp\oplus\ppp^*$ given by $\omega(\ppp,\ppp)=\omega(\ppp^*,\ppp^*)=\lbrace 0 \rbrace$ and
$$\omega(\alpha,v)=\alpha(v)\qquad \forall \alpha \in \ppp^*,~\forall v\in \ppp$$
In the Weyl algebra $\CW(\ppp\oplus\ppp^*,\omega)$ we have $[z_i^*,z_j]=\delta_{ij}$. There is a Lie algebra morphism $\nu':\so(\ppp,B)\rightarrow \CW(\ppp\oplus\ppp^*,\omega)$ given by
$$\nu'(f)=\sum\limits_if(z_i)z_i^*\qquad\forall f\in \so(\ppp,B).$$
\begin{rem}
If we extend $f\in\so(\ppp,B)$ to $f\in {\rm End}(S(\ppp))$ then we have
$$m(\nu'(f))=f \qquad \forall f\in\so(\ppp,B).$$
\end{rem}
Hence, we have a Lie algebra morphism $\nu:\kk\rightarrow \CW(\ppp\oplus\ppp^*,\omega)$ given by $\nu=\nu'\circ \ad|_{\ppp}$ and a Lie algebra morphism $\Delta:\kk\rightarrow U(\kk)\otimes \CW(\ppp\oplus\ppp^*,\omega)$ given by
$$\Delta(x)=x\otimes 1+1\otimes \nu(x)\qquad\forall x\in \kk.$$
\begin{rem}
We have
\begin{equation*}
    [D^+,\Delta(x)]=[D^-,\Delta(x)]=0 \qquad \forall x\in \kk.
\end{equation*}
\end{rem}
\vspace{0.2cm}

\begin{thm}\label{t:comm}
In $U(\gg)\otimes \CW(\ppp\oplus\ppp^*,\omega)$, we have
\begin{align*}
     [D^+,D^-]&=\Big(-\Omega(\gg)+\Omega(\kk)\Big)\otimes 1-\sum_k w_k\otimes\nu(w^k)\\
     &=\Big(-\Omega(\gg)+\frac{3}{2}\Omega(\kk)\Big)\otimes 1-\frac{1}{2}\Delta(\Omega(\kk))+\frac{1}{2}\cdot 1\otimes \nu(\Omega(\kk)).
     \end{align*}
\end{thm}
\begin{proof}
We have
\begin{align*}
    [D^+,D^-]&=\sum\limits_{i,j}\Big(z_iz_j\otimes z^iz_j^*-z_jz_i\otimes z_j^*z^i\Big)\\
    &=\sum\limits_{i,j}\Big(z_iz_j\otimes z^iz_j^*-z_jz_i\otimes (z^iz_j^*+\omega(z_j^*,z^i))\Big)\\
    &=\sum\limits_{i,j}\Big([z_i,z_j]\otimes z^iz_j^*-z_jz_iz_j^*(z^i)\otimes 1\Big).
\end{align*}
We have
$$\sum\limits_{i,j}z_jz_iz_j^*(z^i)=\sum_iz^iz_i=\Omega(\gg)-\Omega(\kk).$$
On the other hand
\begin{align*}
    \sum\limits_{i,j}[z_i,z_j]\otimes z^iz_j^*&=\sum\limits_{i,j,k}B([z_i,z_j],w^k)w_k\otimes z^iz_j^*=\sum\limits_{k}w_k\otimes\sum\limits_{i,j}B([z_i,z_j],w^k) z^iz_j^*\\
    &=-\sum\limits_{k}w_k\otimes\sum\limits_{i,j}B([w^k,z_j],z_i) z^iz_j^*
    =-\sum\limits_{k}w_k\otimes\sum\limits_{j}[w^k,z_j]z_j^*\\
    &=-\sum\limits_{k}w_k\otimes\nu(w^k)
    \end{align*}
and so
\begin{equation*}
    [D^+,D^-]=-\sum\limits_{k}w_k\otimes\nu(w^k)-\Big(\Omega(\gg)-\Omega(\kk)\Big)\otimes 1.
\end{equation*}
Since
\begin{align*}
    \Delta(\Omega(\kk))&=\sum\limits_{k}\Big(w_k\otimes 1+1\otimes \nu(w_k)\Big)\Big(w^k\otimes 1+1\otimes \nu(w^k)\Big)\\
    &=\Omega(\kk)\otimes 1+2\sum\limits_{k}w_k\otimes\nu(w^k)+1\otimes \nu(\Omega(\kk)) 
\end{align*}
then
\begin{equation*}
    [D^+,D^-]=\Big(-\Omega(\gg)+\frac{3}{2}\Omega(\kk)\Big)\otimes 1-\frac{1}{2}\Delta(\Omega(\kk))+\frac{1}{2}\cdot 1\otimes \nu(\Omega(\kk)).
\end{equation*}
\end{proof}

\subsection{The embedding of the Casimir element in the Weyl algebra}
We now study the element $\nu(\Omega(\kk))$. Remark that from \cite[Proposition 2.6]{Ko2}
$$\eta(\nu(\Omega(\kk)))\in S^4(\ppp\oplus\ppp^*)\oplus S^0(\ppp\oplus\ppp^*).$$

\begin{pp}
We have
\begin{equation*}
    \Big(\eta(\nu(\Omega(\kk)))\Big)_0=\frac{1}{12}Tr\Big(ad_{\kk}(\Omega(\kk))-ad_{\gg}(\Omega(\gg))\Big).
\end{equation*}
\end{pp}
\begin{proof}
Using Equation \eqref{e:eta}, We have
\begin{align*}
    \Big(\eta(\nu(\Omega(\kk)))\Big)_0&=\sum_k \Big(\gamma([w^k,z_i]z_i^*[w_k,z_j]z_j^*)(1)\Big)_0=\frac{1}{4}\sum_{i,j,k} \omega(z_i^*,[w_k,z_j])\omega([w^k,z_i],z_j^*)\\
    &=-\frac{1}{4}\sum_{i,j,k}B(z^i,[w_k,z_j])B(z^j,[w^k,z_i])=-\frac{1}{4}\sum_{j}B(z^j,ad(\Omega(\kk))(z_j)).
\end{align*}
On the other hand, we have
\begin{align*}
    Tr\Big(ad_{\gg}(\Omega(\gg))-&ad_{\kk}(\Omega(\kk))\Big)\\
    &=\sum_{k}B(w^k,ad(\Omega(\gg))(w_k))+\sum_i B(z^i,ad(\Omega(\gg))(z_i))-\sum_k B(w^k,ad(\Omega(\kk))(w_k))\\
    &=\sum_{k}B(w^k,ad(\sum_i z^iz_i)(w_k))+\sum_i B(z^i,ad(\Omega(\kk))(z_i))+\sum_i B(z^i,ad(\sum_j z^jz_j)(z_i)).
\end{align*}
Since
\begin{align*}
    \sum_{k}B(w^k,ad(\sum_i z^iz_i)(w_k))&=\sum_{i,k}B([w^k,z^i],[z_i,w_k])=-\sum_{i,k}B([[w^k,z^i],w_k],z_i)\\
    &=\sum_{i,k}B([w_k,[w^k,z^i]],z_i)=\sum_{i,k}B(z_i,ad(\Omega(\kk))(z^i))
\end{align*}
and
\begin{align*}
    \sum_{i}B(z_i,ad(\sum_j z^jz_j)(z^i))&=\sum_{i,j}B(z_i,[z^j,[z_j,z^i]])=\sum_{i,j}B(z_i,[[z^i,z_j],z^j])\\
    &=\sum_{i,j}B([z^i,z_j],[z^j,z_i])=\sum_{i,j,k}B(w^k,[z^i,z_j])B([z^j,z_i],w_k)\\
    &=-\sum_{i,j,k}B([w^k,z^i],z_j])B(z_i,[z^j,w_k])=-\sum_{j,k}B([w^k,[z^j,w_k]],z_j])\\
    &=\sum_jB(z_j,ad(\Omega(\kk))(z^j)),
\end{align*}
we obtain
$$Tr\Big(ad_{\gg}(\Omega(\gg))-ad_{\kk}(\Omega(\kk))\Big)=3\sum_i B(z^i,ad(\Omega(\kk))(z_i))$$
and so
$$\Big(\eta(\nu(\Omega(\kk)))\Big)_0=\frac{1}{12}Tr\Big(ad_{\kk}(\Omega(\kk))-ad_{\gg}(\Omega(\gg))\Big).$$
\end{proof}

Using the strange Freudenthal-de Vries formula \cite{FdV} we can express this constant for reductive Lie algebras as follows:

\begin{cor}\label{c:k-0}
Suppose that $k$ is algebraically closed and suppose that $\kk$ and $\gg$ are reductive. We have
\begin{equation*}
\Big(\eta(\nu(\Omega(\kk)))\Big)_0=2\Big(B(\rho_{\kk},\rho_{\kk})-B(\rho_{\gg},\rho_{\gg})\Big),
\end{equation*}
where $\rho_{\gg}$ (resp. $\rho_{\kk}$) is the Weyl vector of $\gg$ (resp. $\kk$) with respect to a Cartan subalgebra $\hh_{\gg}$ (resp. $\hh_{\kk}$) of $\gg$ (resp. $\kk$) and a choice of a full set of positive roots for the action of $ad(\hh_{\gg})$ (resp. $ad(\hh_{\kk})$).
\end{cor}

\begin{proof}
Let $(\uu,B_{\uu})$ be a quadratic reductive Lie algebra. Let $\hh_{\uu}$ be a Cartan subalgebra of $\uu$ and let $\rho_{\uu}$ be the Weyl vector of $\uu$ with respect to a choice of a full set of positive roots for the action of $ad(\hh_{\uu})$. We have
$$Tr\Big(ad(\Omega(\uu))\Big)=24B_{\uu^*}(\rho_{\uu},\rho_{\uu})$$
which is a general formulation of the strange Freudenthal-de Vries formula, see \cite[Proposition 1.84]{Ko2}.
\end{proof}

Recall from Kostant \cite{Ko3} the following result:

\begin{pp}\label{p:Kostant}
Let $\rho : \hh \rightarrow \sp(V,\omega)$ be a finite-dimensional symplectic representation of a finite-dimensional quadratic Lie algebra $(\hh,B)$ and let $\mu : S^2(V)\rightarrow\hh$ be the map given by
\begin{equation*}
        B(x,\mu(v,w))=\omega(\rho(x)(v),w)\qquad \forall x\in \hh, ~ \forall v,w\in V.
\end{equation*}
Consider the Lie algebra morphism $\nu : \hh \rightarrow \CW(V,\omega)$.
\begin{enumerate}[label=\alph*)]
\item If there exists a quadratic Lie superalgebra structure on $(\hh\oplus V,B\perp \omega)$ extending the bracket of $\hh$ and the action of $\hh$ on $V$, then
$$\lbrace v,w \rbrace=\mu(v,w) \qquad \forall v,w\in V.$$
\item Let $\hht:=\hh\oplus V$, let $B_{\hht}:=B\perp \omega$ and let $\lbrace \phantom{v},\phantom{v} \rbrace : \hht\times\hht\rightarrow\hht$ be the unique $\mathbb{Z}_2$-graded super-antisymmetric bilinear map which extends the bracket of $\hh$, the action of $\hh$ on $V$ and such that
$$\lbrace v,w\rbrace=\mu(v,w) \qquad \forall v,w \in V.$$
Then the following are equivalent:
\begin{enumerate}[label=\roman*)]
\item $(\hht,B_{\hht},\lbrace \phantom{v},\phantom{v} \rbrace)$ is a quadratic Lie superalgebra.
\item $\Big(\eta(\nu(\Omega(\hh)))\Big)_4=0\in S^4(V)$.
\end{enumerate}
\end{enumerate}
\end{pp}
\vspace{0.1cm}

It follows from this proposition that the element $(\eta(\nu(\Omega(\kk))))_4\in S^4(\ppp\oplus\ppp^*)$ is the obstruction to have a quadratic Lie superalgebra structure on $\kk\oplus(\ppp\oplus\ppp^*)$ which extends the bracket of $\kk$ and the action of $\kk$ on $\ppp\oplus\ppp^*$. In particular:
\vspace{0.1cm}

\begin{pp}\label{p:Casimirdegree4nonzero}
If $B([\ppp,\ppp],[\ppp,\ppp])\neq \lbrace 0\rbrace$ then we have
\begin{equation*}
    \Big(\eta(\nu(\Omega(\kk)))\Big)_4\neq 0\in S^4(\ppp\oplus \ppp^*).
\end{equation*}
\end{pp}
\begin{proof}
Consider the map $\mu:S^2(\ppp\oplus\ppp^*)\rightarrow \kk$ given by
\begin{equation*}
    B(x,\mu(v,w))=\omega(x(v),w)\qquad \forall x\in \kk,~ \forall v,w\in \ppp\oplus\ppp^*.
\end{equation*}
\begin{lm}
Let $v,w\in \ppp,~\alpha,\beta \in \ppp^*$ and let $a\in \ppp$ be the unique element such that $\alpha=B(a,\phantom{v})$. We have
\begin{equation*}
    \mu(v,w)=\mu(\alpha,\beta)=0, \qquad \mu(v,\alpha)=[a,v].
\end{equation*}
\end{lm}
\begin{proof}
Since the vector space $\ppp$ is stable under the action of $\kk$ and isotropic for the symplectic form $\omega$, then we have $\mu(v,w)=0$. Similarly we have $\mu(\alpha,\beta)=0$. For $x\in \kk$ we have
\begin{align*}
    B(x,\mu(v,\alpha))=\omega(x(v),\alpha)=-\alpha(x(v))=-B(a,x(v))=B(x(a),v)=B(x,[a,v])
\end{align*}
and so $\mu(v,\alpha)=[a,v]$.
\end{proof}

By Proposition \ref{p:Kostant} we have $(\eta(\nu(\Omega(\kk))))_4=0$ if and only if
\begin{equation}\label{e:proofdegree4nonzero}
    \mu(u,v)(w)+\mu(v,w)(u)+\mu(w,u)(v)=0 \qquad \forall u,v,w\in \ppp\oplus\ppp^*.
\end{equation}
Let $a,v\in \ppp$ and $\alpha:=B(a,\phantom{v})\in \ppp^*$. We have
$$B\Big(\mu(v,\alpha)(v)+\mu(\alpha,v)(v)+\mu(v,v)(\alpha),a\Big)=2B([[a,v],v],a)=-2B([a,v],[a,v]).$$
Hence, if \eqref{e:proofdegree4nonzero} holds, then
$$B([a,v],[a,v])=0 \qquad \forall a,v\in\ppp$$
and by polarisation its implies
$$B([\ppp,\ppp],[\ppp,\ppp])=\lbrace 0\rbrace$$
which is a contradiction.
\end{proof}

\begin{rem}
Suppose that $\kk=\so(\ppp)$ and $\kk\rightarrow \so(\ppp)$ is the natural representation. Hence $\gg\cong\so(\ppp\oplus L)$ where $L$ is a one-dimensional quadratic vector space.

Using Propositions \ref{p:Kostant} and \ref{p:Casimirdegree4nonzero}, there is no quadratic Lie superalgebra structure on $\so(\ppp)\oplus(\ppp\oplus\ppp^*)$ which extends the bracket of $\so(\ppp)$ and the action of $\so(\ppp)$ on $\ppp\oplus\ppp^*$. However, the Lie algebra $\sl(2,k)$ acts on $\ppp\oplus\ppp^*$ by
$$H(z_i)=-z_i,\quad H(z_i^*)=z_i^*,\quad X(z_i)=z_i^*,\quad X(z_i^*)=0,\quad Y(z_i)=0,\quad Y(z_i^*)=z_i$$
where $\lbrace X,H,Y\rbrace$ is the standard basis of $\sl(2,k)$ and
there is a quadratic Lie superalgebra structure on
$$\so(\ppp)\oplus\sl(2,k)\oplus(\ppp\oplus\ppp^*).$$
This Lie superalgebra is isomorphic to the orthosymplectic Lie superalgebra $\osp((\ppp,B|_{\ppp})\perp(k^2,\omega_{k^2}))$ where $\omega_{k^2}$ is the canonical symplectic form over $k^2$. In other words, the natural representation of $\so(\ppp)$ is orthogonal special in the sense of \cite{Mey}. In particular, using Proposition \ref{p:Kostant}, we have that
$$\Big(\eta(\nu(\Omega(\so(\ppp))))\Big)_4=-\Big(\eta(\nu(\Omega(\sl(2,k))))\Big)_4\in \CW(\ppp\oplus\ppp^*,\omega).$$
\end{rem}

\subsection{Unitary structures} Suppose now that $k=\mathbb C$ and that $\mathfrak g=\mathfrak k\oplus\mathfrak p$ is the complexification of a real Lie algebra $\mathfrak g_0=\mathfrak k_0\oplus \mathfrak p_0$. Let $\overline\ $ denote the complex-conjugation on $\mathfrak g$ whose real points is $\mathfrak g_0$. 
Define a star operation (conjugate-linear anti-involution) on $\CW=\CW(\ppp+\ppp^*)$ by 
\begin{equation}
v^*=\iota(v),\quad \text{for all }v\in \ppp+\ppp^*,
\end{equation}
and extended as an anti-homomorphism. Then $\CP=S(\ppp)$ is naturally a (pre)unitary module for $\CW$ with respect to $*$ with the hermitian pairing:
\begin{equation}
\langle P_1,P_2\rangle_\CP=\partial_{P_2}(\overline P_1), \text{ for all }P_1,P_2\in\CP,
\end{equation}
where 
$\partial_{P_2}$ is the partial differential operator defined by $P_2$.

Define a star operation on $U(\gg)$ by extending as a conjugate-linear anti-homomorphism the assignment
\begin{equation}
x^*=-\overline x,\text{ for all }x\in \gg.
\end{equation}

\begin{lm}
In $U(\gg)\otimes\CW$, $(D^\pm)^*=-D^{\mp}$.
\end{lm}

\begin{proof}
Straightforward.
\end{proof}

\begin{rem}
Another important star operation (the "compact" star operation) on $U(\gg)$ was considered in \cite{ALTV} in the study of unitarisable Harish-Chandra modules. This is defined on $\gg$ as follows:
\begin{equation}
x^c=-\overline x,\ \text{ for }x\in \kk,\quad z^c=\overline z,\ \text{ for } z\in \ppp.
\end{equation}
If we extend this to a star operation $c$ of $U(\gg)\otimes \CW$ (by $c$ on $U(g)$ and the same $*$ as before on $\CW$), then it is immediate that
\begin{equation}
(D^\pm)^c=D^\mp.
\end{equation} 
The discussion below with respect to the classical $*$ can be easily modified for $c$ as well.
\end{rem}

We can consider $D^\pm$ as operators on $M\otimes \CP$ for every $U(\gg_\bC)$-module $(\pi, M)$.

Suppose $(\pi,M)$ is an admissible $(\gg_\bC,K)$-module which admits a nondegenerate hermitian form $\langle~,~\rangle_M$ invariant under $*$. Define the product form 
\begin{equation}
\langle~,~\rangle_{M\otimes\CP}=\langle~,~\rangle_{M}\langle~,~\rangle_{\CP}.
\end{equation}

\begin{pp}\label{p:diff}
Suppose that $M$ admits an infinitesimal character $\chi_M$. Let $\sigma$ be a simple finite-dimensional $\kk$-module and let $(M\otimes\CP)(\sigma)$ denote the $\Delta(\kk)$-isotypic component of $\sigma$. If $x\in (M\otimes \C P)(\sigma)$, then
\begin{align*} \langle D^+ x, D^+ x\rangle_{M\otimes\CP}-\langle D^- x, D^- x\rangle_{M\otimes\CP}&=\left(-\chi_M(\Omega(\gg))-\frac 12 \sigma(\Omega(\kk))\right)\langle x,x\rangle_{M\otimes\CP}\\
&+\frac 32 \langle (\pi(\Omega(\kk))\otimes 1)x,x\rangle_{M\otimes\CP}+\frac 12 \langle x, (1\otimes\nu(\Omega(\kk)))x \rangle_{M\otimes\CP}.
\end{align*}
\end{pp}

\begin{proof}
This follows immediately from Theorem \ref{t:comm} using the adjointness property $(\C D^{\pm})^*=-\C D^{\mp}.$
\end{proof}

As a particular example, notice that when $x\in M(\sigma)\otimes S^0(\ppp)=M(\sigma)\otimes 1$, then $D^- x=0=(1\otimes\nu(\Omega_\kk)x$, hence we recover the well-known "Casimir inequality" for unitary modules:

\begin{cor}\label{c:Casimir}
For all $x\in M(\sigma)\otimes 1$:
\[\langle \C D^+ x, \C D^+ x\rangle_{M\otimes\C P}=(-\chi_M(\Omega(\gg))+\sigma(\Omega(\kk)))\langle x,x\rangle_{M\otimes \C P}.\]
Moreover, if $M$ is $*$-unitary then $\chi_M(\Omega(\gg))\le \sigma(\Omega(\kk))$, for all simple finite dimensional $\kk$-modules $\sigma$ such that $M(\sigma)\neq 0$.

\end{cor}

\begin{proof}
The inequality follows by taking $x\neq 0$ in $M(\sigma)\otimes 1$ and using that $\langle \C D^+ x, \C D^+ x\rangle_{M\otimes\C P}\ge 0$ for unitary modules. Notice that $(\pi(\Omega(\kk))\otimes 1)x=(\sigma(\Omega(\kk)))x$ in this case.
\end{proof}

Suppose $p\in S(\ppp_0)$. Notice that for all $w\in \kk$, 
\[\langle \nu(w)p,p\rangle_\CP=0.
\]
This is because 
\begin{align*}
\langle\nu(w)p,p\rangle_\CP&=\sum_{i,j} B([w,z_i],z_j)\langle z^i z_j^*p,p\rangle=\sum_{i,j} B([w,z_i],z_j)\langle z_j^*p,z_i^* p\rangle_\CP\\&=\sum_i B([w,z_i],z_i)\langle z_i^*p,z_i^* p\rangle_\CP=0,
\end{align*}
 using that $ \langle z_j^*p,z_i^* p\rangle_\CP=0$ if $i\neq j.$
Therefore $\langle \pi(w)\otimes \nu(w) x,x\rangle_{M\otimes \CP}=0$ as well for every simple tensor $x\in M\otimes S(\ppp_0)$. Using Theorem \ref{t:comm}, we obtain immediately:

\begin{cor}Suppose $x\in M(\sigma)\otimes S(\ppp_0)$ is a simple tensor. Then 
\[\langle D^+ x, D^+ x\rangle_{M\otimes\CP}-\langle D^- x, D^- x\rangle_{M\otimes\CP}=\left(-\chi_M(\Omega(\gg))+\sigma(\Omega(\kk))\right)\langle x,x\rangle_{M\otimes\CP}.
\]
\end{cor}

\begin{rm}
The structure of $S(\ppp)$ as a $\kk$-module is well known by the theorem of Kostant and Rallis \cite{KR}. Suppose $G$ is a complex linear algebraic group with Lie algebra $\gg$ and let $\theta:G\to G$ be a regular involution such that $K=G^\theta$ has Lie algebra $\kk$. Then $\ppp$ is the $(-1)$-eigenspace of the differential of $\theta$ on $\gg$. Let $\mathfrak a$ be a Cartan subspace of $\ppp$ and set $M=Z_K(\mathfrak a)$. The subspace of $K$-invariants $S(\ppp)^K$ is a polynomial ring, and denote by $S(\ppp)^K_+$ the subspace of $K$-invariant polynomials with zero constant term. Let $I(\ppp)$ denote the ideal of $S(\ppp)$ generated by $S(\ppp)^K_+$. For every degree $\ell\ge 0$, $I(\ppp)\cap S^\ell(\ppp)$, being $K$-invariant, has a unique $K$-invariant complement in $S^\ell(\ppp)$, denoted $\mathcal H^\ell(\ppp)$. Set $\mathcal H(\ppp)=\oplus_{\ell\ge 0} \mathcal H^\ell(\ppp)$. This is the space of $K$-harmonic polynomials in $S(\ppp)$. Then (cf. \cite[\S 12.4.1]{GW}):
\begin{enumerate}
\item $S(\ppp)$ is a free module over $S(\ppp)^K$ and $S(\ppp)\cong S(\ppp)^K\otimes \mathcal H(\ppp)$;
\item As a $K$-representation, $\mathcal H(\ppp)\cong \operatorname{Ind}_M^K(1)$.
\end{enumerate}
$\operatorname{Ind}_M^K(1)$ is the restriction to $K$ of the spherical minimal principal series $G$-representation.
\end{rm}

\subsection{An example: $(\gg,K)$-modules of $SL(2,\bR)$}

Suppose that $\gg$ is $\sl(2,\bC)$ and that $B(x,y)=\frac{1}{2}Tr(xy)$ for all $x,y\in\sl(2,\bC)$. Let $\lbrace X,H,Y \rbrace$ be a $\sl(2,\bC)$-triple and consider the $B$-orthogonal $\mathbb{Z}_2$-gradation $\sl(2,\bC)=\kk\oplus\ppp$ where $\kk={\rm Span}\langle  H\rangle$ and $\ppp={\rm Span}\langle X,Y\rangle$. The Dirac operators $D^+,D^-\in U(\sl(2,\bC))\otimes \CW(\ppp\oplus\ppp^*,\omega)$ satisfy
\begin{align*}
    D^+&=2X\otimes Y+2Y\otimes X,\\
    D^-&=X\otimes X^*+Y\otimes Y^*.
\end{align*}
The Casimir elements associated to $(\kk,B|_{\kk})$ and $(\sl(2,\bC),B)$ satisfy
\begin{align*}
    \Omega(\kk)&=H^2\in U(\kk),\\
    \Omega(\sl(2,\bC))&=H^2+2(XY+YX)\in U(\sl(2,\bC)).
\end{align*}
\begin{rem}
We have
$$\nu(\Omega(\kk))(X^kY^l)=4(k-l)^2X^kY^l \qquad \forall X^kY^l \in S(\ppp).$$
\end{rem}
Let $\lbrace H,X,Y \rbrace$ be the $\sl(2)$-triple defined by
$$H=\begin{pmatrix}
0 & -i\\
i & 0
\end{pmatrix},\qquad X=\frac{1}{2}\begin{pmatrix}
1 & i\\
i & -1
\end{pmatrix},\qquad Y=\begin{pmatrix}
1 & -i\\
i & -1
\end{pmatrix}.$$
Let $\lambda\in\mathbb{C}$ and $\epsilon\in\lbrace 0,1\rbrace$. Let $V_{\lambda,\epsilon}$ be the minimal principal series module \cite{Vo}, defined as vector space with a basis $\lbrace W_n ~ | ~ \forall n\in \mathbb{Z}, ~ n\equiv \epsilon \mod2 \rbrace$ and with an action of $\sl(2,\bC)$ on $V_{\lambda,\epsilon}$ by
\begin{align*}
    \pi(H)(W_n)&=nW_n,\\
    \pi(X)(W_n)&=\frac{1}{2}(\lambda+n+1) W_{n+2},\\
    \pi(Y)(W_n)&=\frac{1}{2}(\lambda-n+1) W_{n-2}.
\end{align*}
We have
$$\Omega(\kk)(W_n)=n^2W_n, \qquad \Omega(\sl(2,\bC))(W_n)=(\lambda^2-1)W_n.$$
\begin{rem}
Let $X^kY^l \in S(\ppp)$. We have
$$\Delta(\Omega(\kk))(W_n\otimes X^kY^l)=\left(n^2+4(k-l)^2+4n(k-l)\right)W_n\otimes X^kY^l.$$
\end{rem}
\begin{pp}\label{p:ps}
Suppose that $\lambda+1$ is not an integer congruent to $\epsilon$ modulo $2$. Consider the representation $\pi\otimes m:\sl(2,\bC)\times \C W\rightarrow {\rm End}(V_{\lambda,\epsilon}\otimes S(\ppp))$. We have 
\begin{align*}
    Ker(\pi\otimes m(D^+))&=\lbrace 0 \rbrace,\\
    Ker(\pi\otimes m(D^-))&={\rm Span}\langle\sum\limits_{i=0}^{l}(-1)^i {l\choose i} \prod\limits_{j=0}^{i-1}x(n+4j)\prod\limits_{j=i+1}^{l}y(n+4j) W_{n+4i}\otimes X^{l-i}Y^i ~ | ~ \forall l\geq 0, ~ \forall n\in \mathbb{Z}\rangle,
\end{align*}
where $x(n)=\frac{1}{2}(\lambda+n+1)$ and $y(n)=\frac{1}{2}(\lambda-n+1)$.
\end{pp}
\begin{proof}
Let $v=\sum\limits_n W_n\otimes P_n\in V_{\lambda,\epsilon}\otimes S^l(\ppp)$.
\vspace{0.2cm}

We first show that $D^+$ is injective. We have
\begin{align*}
    D^+(v)&=2\sum\limits_n (x(n)W_{n+2}\otimes YP_n+y(n)W_{n-2}\otimes XP_n)\\
    &=2\sum\limits_n W_n\otimes(x(n-2)YP_{n-2}+y(n+2)XP_{n+2}).
\end{align*}
We have $P_k=\sum\limits_{j=0}^{l}a_{k,j}X^{l-j}Y^{j}$ and so
\begin{equation*}
    D^+(v)=2\sum\limits_n W_{n}\otimes\Big(a_{n+2,0}y(n+2)X^{l+1}+\sum\limits_{j=1}^l(a_{n-2,j-1}x(n-2)+a_{n+2,j}y(n+2))X^{l-j+1}Y^j+a_{n-2,l}x(n-2)Y^{l+1}\Big).
\end{equation*}
If $D^+(v)=0$, then, for all $n$, $a_{n,0}=a_{n,l}=0$ and
\begin{equation}\label{e:D^+inj}
    a_{n-2,j-1}x(n-2)+a_{n+2,j}y(n+2)=0 \qquad \forall j\in \llbracket 1,l \rrbracket.
\end{equation}
Hence, by induction and using \eqref{e:D^+inj} we obtain $a_{n,j}=0$ for all $n,j$ and so $v=0$.
\vspace{0.2cm}

We now calculate the kernel of $D^-$. We have
$$D^-(v)=\sum\limits_n W_n\otimes\Big(x(n-2)\partial_X(P_{n-2})+y(n+2)\partial_Y(P_{n+2})\Big).$$
We have $P_k=\sum\limits_{j=0}^{l}a_{k,j}X^{l-j}Y^{j}$ and so
$$D^-(v)=\sum\limits_n W_n\otimes\sum\limits_{j=0}^{l-1}\Big(a_{n-2,j}x(n-2)(l-j)+a_{n+2,j+1}y(n+2)(j+1)\Big)  X^{l-j-1}Y^{j}.$$
If $D^-(v)=0$, then, for all $n$, we have
\begin{equation}\label{e:D^-kernel}
    a_{n,j}x(n)(l-j)+a_{n+4,j+1}y(n+4)(j+1) \qquad \forall j\in \llbracket 0,l-1 \rrbracket.
\end{equation}
Hence, by induction and using \eqref{e:D^-kernel} we obtain
$$a_{n+4i,i}=(-1)^i {l\choose i}\frac{\prod\limits_{j=0}^{i-1}x(n+4j)}{\prod\limits_{j=1}^{i}y(n+4j)}.$$
\end{proof}

\begin{cor}
Suppose that $\lambda=k-1$ and let $v=W_k\otimes Y$. We have $v\in Ker(\pi\otimes m(D^-))$ and
$$[D^+,D^-](v)=4(\lambda+1)v.$$
\end{cor}

Let $n\in\mathbb{N}$ and let $V_n$ be the $(n+1)$-dimensional module of $\sl(2,\bC)$, i.e., the vector space with a basis $\lbrace W_k ~ | ~ \forall k\in\llbracket 0,n \rrbracket\rbrace$ and define an action of $\sl(2,\bC)$ on $V_n$ by
\begin{align*}
    \pi(H)(W_k)&=(-n+2k)W_k,\\
    \pi(X)(W_k)&=(n-k) W_{k+1},\\
    \pi(Y)(W_k)&=k W_{k-1}.
\end{align*}

\begin{pp}
Consider the representation $\pi\otimes m:\sl(2,\bC)\times \C W\rightarrow {\rm End}(V_n\otimes S(\ppp))$.
\begin{enumerate}[label=\alph*)]
\item If $n$ is odd, then $\pi\otimes m(D^+)$ is injective. If $n$ is even, then the kernel of $\pi\otimes m(D^+)$ is generated by elements of the form
\begin{equation*}
    \sum\limits_{i=0}^{\frac{n}{2}}(-1)^i {\frac{n}{2}\choose l+i} \prod\limits_{j=0}^{i-1}(n-2j)\prod\limits_{j=i+1}^{\frac{n}{2}}(2j) W_{2i}\otimes X^{k-i}Y^{l+i}
\end{equation*}
where $k,l\geq \frac{n}{2}$.
\item The kernel of $\pi\otimes m(D^-)$ is generated by elements of the form 
\begin{equation*}
    \sum\limits_{i=0}^{min(n',m-l)}(-1)^i {m\choose l+i} \prod\limits_{j=0}^{i-1}(n-2j-k)\prod\limits_{j=i+1}^{m-l}(2j+k) W_{k+2i}\otimes X^{m-l-i}Y^{l+i}
\end{equation*}
where $m\geq 0$, $k\in\llbracket 0,n \rrbracket$, $l\in\llbracket 0,m \rrbracket$ such that $k=0$ or $l=0$ and $n-k=2n'+\epsilon ~ (\epsilon\in\lbrace 0,1\rbrace)$.
\end{enumerate}
\end{pp}

\begin{proof}
Let $v=\sum\limits_{i=0}^n W_i\otimes P_i\in V_n\otimes S^l(\ppp)$.
\vspace{0.1cm}

\noindent
$a)$ We have
\begin{equation}\label{e:D+ sl2 finite dim rep}
    D^+(v)=2W_0\otimes XP_1+2W_n\otimes YP_{n-1}+2\sum\limits_{i=1}^{n-1}W_i\otimes\Big((n-i+1)YP_{i-1}+(i+1)XP_{i+1}\Big).
\end{equation}
Suppose that $D^+(v)=0$. In particular $P_1=0$ and, using \eqref{e:D+ sl2 finite dim rep}, this implies $P_i=0$ for all $i$ odd. Similarly, if $n$ is odd, then $P_{n-1}=0$ implies that $P_i=0$ for all $i$ even and so $v=0$. If $n$ is even, the formula follows from a computation similar to Proposition \ref{p:ps}.
\vspace{0.1cm}

\noindent
$b)$ We have
\begin{equation}\label{e:D- sl2 finite dim rep}
    D^-(v)=W_0\otimes \partial_Y(P_1)+W_n\otimes \partial_X(P_{n-1})+\sum\limits_{i=1}^{n-1}W_i\otimes\Big((n-i+1)\partial_X(P_{i-1})+(i+1)\partial_Y(P_{i+1})\Big).
\end{equation}
Suppose $D^-(v)=0$ and $P_i=\sum\limits_{j=0}^{l}a_{i,j}X^{l-j}Y^{j}$. We have $\partial_Y(P_1)=\partial_X(P_{n-1})=0$ and
$$\sum\limits_{j=1}^{l} (i+1)ja_{i+1,j}+(n-i+1)(l-j-1)a_{i-1,j-1}=0,$$
and the formula follows from a computation similar to Proposition \ref{p:ps}.
\end{proof}

\section{Graded Hecke algebras}

\subsection{Drinfeld's degenerate Hecke algebra}\label{s:DrinfeldA} Let $k$ be a field of characteristic $0$. (As before, the interesting cases for us will be $k=\bR$ or $\bC$.) Let $V$ be a finite-dimensional vector space with a nondegenerate symmetric bilinear form $B$. Let $\Gamma$ be finite subgroup of $SO(V,B)$. Suppose that we have a family of skew-symmetric forms on $V$, $(a_\gamma)$, $\gamma\in \Gamma$. Define the associative unital algebra $\mathbf H$ as the quotient of the smash-product algebra $T(V)\# k[\Gamma]$ by the relations
\begin{equation}
[v_1,v_2]=\sum_{\gamma\in\Gamma}a_\gamma(v_1,v_2) \gamma,\quad v_1,v_2\in V.
\end{equation}
Of interest are those algebras $\mathbf H$ which have a PBW property, i.e., the associated graded algebra with respect to the filtration where $V$ gets degree $1$ and $\Gamma$ gets degree $0$ is naturally isomorphic to the algebra $S(V)\# k[\Gamma]$.

As before, let $\CW=\CW(V\oplus V^*,\omega)$ be the Weyl algebra and define the symplectic Dirac elements in $\mathbf H\otimes \CW$:
\begin{equation}
D^-=\sum_i v_i\otimes v_i^*,\quad D^+=\sum_i v_i\otimes v^i,
\end{equation}
where $\{v_i\}$ is a basis of $V$, $\{v^i\}$ is the $B$-dual basis of $V$, and $\{v_i^*\}$ is the dual basis in $V^*$ to $\{v_i\}$.

Let $\psi: ({\bigwedge}^2V)^*\to {\bigwedge}^2V$ be the identification given by $B$. Under this, $a_\gamma$ corresponds to 
\[a_\gamma\mapsto \sum_{i,j} a_\gamma(v_i,v_j) v^i\wedge v^j.
\]
Next, we may identify $\mu': {\bigwedge}^2V\to \mathfrak{so}(V,B)$ by
\[\mu'(u\wedge v)(w)=B(u,w)v-B(v,w)u.
\]
Finally, recall the Lie algebra morphism $\nu': \mathfrak{so}(V,B)\to \CW(V\oplus V^*,\omega)$, $\nu'(f)=\sum_i f(v_i) v_i^*$. Set 
\begin{equation}
\tau(\gamma)=\nu'\circ\mu'\circ \psi(a_\gamma).
\end{equation}
We compute $\tau_\gamma$ explicitly. 
\begin{align*}
\tau(\gamma)&=\nu'\circ\mu' (\sum_{i,j} a_\gamma(v_i,v_j) v^i\wedge v^j)= \sum_{i,j}  a_\gamma(v_i,v_j) \nu'(\mu'(v^i\wedge v^j))\\
&=\sum_{i,j} a_\gamma(v_i,v_j) \sum_\ell \mu'(v^i\wedge v^j)(v_\ell) v_\ell^*=\sum_{i,j} a_\gamma(v_i,v_j) \sum_\ell (B(v^i,v_\ell) v^j v_\ell^*- B(v^j,v_\ell) v^i v_\ell^*)\\
&=\sum_{i,j} a_\gamma(v_i,v_j) (v^j v_i^*-v^i v_j^*)
\end{align*}
Hence
\begin{equation}\label{e:tau}
\tau(\gamma)=-2\sum_{i,j} a_\gamma(v_i,v_j) v^i v_j^*.
\end{equation}

\begin{pp}\label{p:Drinfeld}
In $\mathbf H\otimes\CW$, 
\[ [D^+,D^-]=-\Omega_V\otimes 1-\frac 12\sum_{\gamma\in \Gamma} \gamma\otimes \tau(\gamma),
\]
where $\Omega_V=\sum_i v^i v_i\in \mathbf H^\Gamma.$
\end{pp}

\begin{proof}
The calculation is similar to that in the Lie algebra case. 
\begin{align*}
[D^+,D^-]&=\sum_{i,j} (v_i v_j\otimes v^i v_j^*-v_j v_i\otimes v_j^* v^i)\\
&=\sum_{i,j} [v_i,v_j]\otimes v^i v_j^*-\sum_{i,j} v_j v_i v_j^*(v^i)\otimes 1\\
&=-\sum_i v^i v_i +\sum_\gamma \gamma\otimes \sum_{i,j} a_\gamma(v_i,v_j) v^i v_j^*,
\end{align*}
and the claim follows from (\ref{e:tau}).
\end{proof}

\subsection{The graded affine Hecke algebra}
Let $(V_0^*,\Phi,V_0,\Phi^\vee)$ be a real root system, where $V_0$ is a finite-dimensional real vector space, $V_0^*$ its dual, $\Phi\subset V_0^*$ the set of roots, and $\Phi^\vee\subset V_0$ the set of coroots. Let $W$ be the finite Weyl group and $B$ be a positive-definite $W$-invariant symmetric bilinear form on $V_0$. Denote $V=\bC\otimes_\bR V_0$ and $V^*=\bC\otimes_\bR V_0^*$ the complexified vector spaces and extend $B$ to a symmetric bilinear form on $V$. By abuse of notation, denote also by $B$ the dual form on $V_0^*$ and similarly the bilinear extension to $V$.

Fix a choice of positive roots $\Phi^+$ and let $\Pi$ be the corresponding set of simple roots. Let $s_\al\in W$ denote the reflection corresponding to $\al$. Let $T(V)$ denote the tensor algebra of $V$.

The graded affine Hecke algebra  $\bH=\bH(V,\Phi,k)$ attached to this root system and to the $W$-invariant parameter function $k:\Phi\to\bC$ is the associative unital algebra which is the quotient of the smash product algebra $T(V)\rtimes\bC[W]$ by the relations:
\begin{equation}
v\cdot s_\al-s_\al\cdot s_\al(v)=k_\al \al(v),\text{ for all }v\in V,\ \al\in\Pi.
\end{equation}
For every $v\in V$, define
\begin{equation}
T_v=\frac 12\sum_{\al>0}k_\al \al(v) s_\al,\quad \wti v=v-T_v.
\end{equation}
The presentation of $\bH$ as a Drinfeld Hecke algebra is with the generators $w\in W$ and $\wti v$, $v\in V$, via:
\begin{equation}
[\wti v_i,\wti v_j]=[T_{v_j},T_{v_i}]=\frac 14\sum_{\alpha,\beta>0}k_\alpha k_\beta (\alpha(v_j) \beta(v_i) -\al(v_i)\beta(v_j)) s_\al s_\beta
\end{equation}
and 
\[w\wti v w^{-1}=\wti{w(v)}.
\]
In particular, the skew-symmetric forms $a_w$ are $0$ unless $w$ is a product of two distinct reflections and in that case 
\[a_w(\wti v_i,\wti v_j)=\frac 14\sum_{\al,\beta>0,~w=s_\al s_\beta} k_\al k_\beta  (\alpha(v_j) \beta(v_i) -\al(v_i)\beta(v_j)). 
\]
The symplectic Dirac elements are
\[D^-=\sum_i \wti v_i\otimes v_i^*,\quad D^+=\sum_i \wti v_i \otimes v^i.
\]
Notice that since $(s_\al s_\beta)^{-1}=s_\beta s_\al$ in $SO(V,B)$, the commutator $[s_\al,s_\beta]\in \bC[W]$ lies in fact in $\mathfrak{so}(V,B)$. Moreover, these commutators
span $\mathfrak{so}(V,B)$.

\begin{pp}\label{p:span}
The set $S = \{[s_\alpha,s_\beta]\mid \alpha,\beta \in \Phi\}\subseteq \mathfrak{so}(V,B)$ spans $\mathfrak{so}(V,B)$.
\end{pp}

\begin{proof}
There is an identification of $\mathfrak{so}(V,B)$ with the irreducible $W$-representation 
$\bigwedge^2(V)$ under which the action of $W$ on $\mathfrak{so}(V,B)$ is via conjugation. 
Thus, the set $S$ spans a non-zero $W$-invariant subspace of $\mathfrak{so}(V,B)$ which
must be the whole space, by irreducibility. 
\end{proof}

 We can then compute the image under $\nu': \mathfrak{so}(V,B)\to\CW$ of the commutators. 

\begin{lm}\label{l:comm}
Regarding $[s_\al,s_\beta]$ in $\mathfrak{so}(V,B)$, we have
\begin{align*}
\nu'([s_\al,s_\beta])&=\al(\beta^\vee)\beta\al^\vee-\beta(\al^\vee) \al\beta^\vee\\
&=\al(\beta^\vee)\al^\vee\beta-\beta(\al^\vee) \beta^\vee\al 
\end{align*}
in $\CW(V\oplus V^*,\omega).$
\end{lm}
\begin{proof}
For every $v\in V$, notice that $[s_\al,s_\beta](v)=\al(\beta^\vee) \beta(v) \al^\vee-\beta(\al^\vee)\al(v) \beta^\vee.$ Then $\nu'([s_\al,s_\beta])=\sum_i (\al(\beta^\vee) \beta(v_i) \al^\vee-\beta(\al^\vee)\al(v_i) \beta^\vee)v_i^*$ and the first formula follows. The second is immediate from the commutation relations in $\CW$.
\end{proof}

It is also straight-forward to compute 
\begin{equation}
\tau(w)=-\frac 12\sum_{\al,\beta>0~w=s_\al s_\beta} k_\al k_\beta (\iota(\beta)\al-\iota(\al)\beta),
\end{equation}
where $\iota(f)$ is defined via $f(v) = B(\iota(f),v)$ for all $v\in V$. In particular, for any root $\alpha$,
\[\al^\vee = \frac 2{B(\al,\al)}\iota(\al).
\]

\begin{rem}
From Lemma \ref{l:comm}, (and the fact that $B(\alpha,\beta) = B(\iota(\alpha),\iota(\beta))$, it follows that the element
\[
B(\al^\vee,\beta^\vee)^{-1}\nu'([s_\alpha,s_\beta]) =(\iota(\al)\beta-\iota(\beta)\al)\in \CW
\]
is well-defined and non-zero, even if $B(\al^\vee,\beta^\vee)=0$.
\end{rem}

Stretching the analogy with the Lie algebra case, to each pair of positive roots, we denote by
$\Delta([s_\alpha,s_\beta]) = 
[s_\alpha,s_\beta]\otimes 1 + 1 \otimes \nu'([s_\alpha,s_\beta]) \in \bH\otimes \CW$. 
Note that
\[
\Delta([s_\alpha,s_\beta]^2) = [s_\alpha,s_\beta]^2\otimes 1
+2[s_\alpha,s_\beta]\otimes\nu'([s_\alpha,s_\beta])
+1\otimes \nu'([s_\alpha,s_\beta])^2.
\]
Let also $\Phi_2^+=\{\{\al,\beta\}\mid \al,\beta\in\Phi^+, \al\neq\beta\}.$ We emphasise that these are unordered pairs of positive roots.

\begin{thm}\label{p:comm}
In $\bH\otimes \CW$,
\begin{align*}
 [D^+,D^-]&=-\Omega_\bH\otimes 1 +\Omega_W\otimes 1 -\frac 14\sum_{\{\al,\beta\}\in \Phi_2^+} k_\al k_\beta \frac 1{B(\al^\vee,\beta^\vee)}[s_\al, s_\beta]\otimes \nu'([s_\al,s_\beta])\\
 &= \left(-\Omega_\bH +\Omega_W + \Omega'_W\right)\otimes 1 - \Delta(\Omega'_W) + 1\otimes\nu'(\Omega'_W),
\end{align*}
where $\Omega_\bH=\sum_i v_i v^i\in Z(\bH)$ and $\Omega_W,\Omega'_W\in \bC[W]^W$ are given by
\begin{align*}
\Omega_W &=\frac 14\sum_{\al,\beta>0}k_\al k_\beta B(\al,\beta) s_\al s_\beta\\
\Omega'_W&=\frac {1}{16}\sum_{\al,\beta>0}k_\al k_\beta \frac{1}{B(\al^\vee,\beta^\vee)} [s_\al, s_\beta]^2.
\end{align*}
\end{thm}

\begin{proof} As noted before, we can write 
\[\nu'([s_\al,s_\beta])=B(\al^\vee,\beta^\vee) (\iota(\al)\beta-\iota(\beta)\al).
\]
From Proposition \ref{p:Drinfeld}, 
\begin{align*}
[D^+,D^-]&=-\sum_{i}\wti v_i \wti{v^i} +\frac 14\sum_{\al,\beta>0} k_\al k_\beta s_\al s_\beta\otimes (\iota(\beta)\al-\iota(\al)\beta)\\
&=(-\Omega_\bH+\Omega_W)\otimes 1+\frac 14\sum_{\{\al,\beta\}\in\Phi^+_2} k_\al k_\beta [s_\al,s_\beta]\otimes(\iota(\beta)\al- \iota(\al)\beta)\quad \text{(using \cite{BCT}) }\\
&=(-\Omega_\bH+\Omega_W)\otimes 1-\frac 14\sum_{\{\al,\beta\}\in\Phi^+_2} k_\al k_\beta 
\frac 1{B(\al^\vee,\beta^\vee)}[s_\al,s_\beta]\otimes\nu'([s_\al,s_\beta]).
\end{align*}
The second identity then follows from 
\[
\Delta(\Omega'_W) = \Omega'_W\otimes 1 + \sum_{\{\al,\beta\}\in\Phi^+_2} \frac14 k_\al k_\beta \frac 1{B(\al^\vee,\beta^\vee)}[s_\al,s_\beta]\otimes\nu'([s_\al,s_\beta]) + 
1\otimes\nu'(\Omega'_W),
\]
and we are done.
\end{proof}

\begin{ex}
Let $\bH(A_1)$ be the Hecke algebra of type $A_1$. Here $V_0=\bR\al^\vee$, $V_0^*=\bR\al$, where $\al$ is the unique simple root. The bilinear form is such that $(\al^\vee,\al^\vee)=2$. The relation is 
\[ v\cdot s+s\cdot v=2k,
\]
where $v\in V$, $s=s_\al$, and $k=k_\al$. Then $\wti v= v-k s$ and $D^+=\wti v\otimes f$, $D^-=\wti v\otimes e$. It follows that
\[
[D^+,D^-]=-\Omega_\bH\otimes 1 + k^2 (1\otimes 1),
\]
where $\Omega_\bH= \frac 12 (\al^\vee)^2$.
\end{ex}

\begin{ex}
Consider the root system of type $A_2$ with simple roots $\{\al,\beta\}$. The bilinear form is normalised such that $B(\al,\al)=B(\beta,\beta)=2$. The third positive root is $\gamma=\al+\beta$. A direct calculation shows that
\begin{equation}
[D^+,D^-]=-\Omega_\bH\otimes 1+ \frac {3 k^2}2(1\otimes 1)+\frac {k^2}4 (s_\al s_\beta+s_\beta s_\al)\otimes 1-\frac {k^2}4 (s_\al s_\beta-s_\beta s_\al)\otimes (\beta\al^\vee-\al\beta^\vee).
\end{equation}
\end{ex}

\begin{lm}
In $\bH\otimes \CW$:
\begin{equation}
[1\otimes \Delta,D^-]=2 D^+,\quad [1\otimes \Delta,D^+]=0,\quad [1\otimes \Delta,[D^+,D^-]]=0.
\end{equation}
\end{lm}

\begin{proof}
Straightforward.
\end{proof}

Suppose $(\pi,M)$ is an $\bH$-module. Then the actions of $D^\pm$ give rise to the symplectic Dirac operators:
\begin{equation}
D^\pm: M\otimes \C P\to M\otimes\C P.
\end{equation}
Notice that $D^+$ maps $M\otimes S^j(V)$ to $M\otimes S^{j+1}(V)$, while $D^-$ maps $S^j(V)$ to $S^{j-1}(V)$.

If $\sigma$ is an irreducible $W$-representation, denote by $M(\sigma)$ the $\sigma$-isotypic component of $M$. Suppose $M$ has a central character $\chi_M$ (e.g., $M$ is a simple module). The central characters of $\bH$-modules are parameterised by $W$-orbits in $V^*$, and we will think implicitly of $\chi_M$ as an element (or a $W$-orbit) in $V^*$. Then $\Omega$ acts on $M$ by a scalar multiple of the identity, where the scalar is
\begin{equation}
\pi(\Omega)=B(\chi_M,\chi_M),
\end{equation}
see \cite{BCT} for example. On the other hand, $\Omega_W$ acts on $M(\sigma)$ by a scalar multiple of the identity, where the scalar is
\begin{equation}
\sigma(\Omega_W)=\frac 14\sum_{\al,\beta>0}k_\al k_\beta B(\al,\beta) \frac{\tr\sigma(s_\al s_\beta)}{\tr \sigma(1)}.
\end{equation}

\begin{ex}
If $\sigma=\triv$ is the trivial $W$-representation, then 
\[\triv(\Omega_W)=B(\rho_k,\rho_k),\text{ where }\rho_k=\frac 12\sum_{\al>0} k_\al \al.
\]
Notice that $\rho_k$ is exactly the central character of the trivial $\bH$-module.
\end{ex}

From Proposition \ref{p:comm}, it follows that, when $M$ has a central character, if $x\in M(\sigma)\otimes \C P$, then
\begin{align}\label{e:iso}
[\C D^+,\C D^-]x&=(-\pi(\Omega)+\sigma(\Omega_W))x - \C E_{\sigma,P} x,\qquad\text{ where}\\
\C E_{\sigma,P}&=\frac 14\sum_{\{\al,\beta\}\in\Phi^+_2} k_\al k_\beta \sigma([s_\al,s_\beta])\otimes \C X_{\al,\beta},\text{ with } \C X_{\al,\beta}=\iota(\al)\partial_\beta-\iota(\beta)\partial_\al.
\end{align}

Notice that every $\C X_{\al,\beta}$ is a differential operator on $\C P$ preserving the degree.

\subsection{The element $\Omega_W$}\label{s:4.3} We look in more detail at the element $\Omega_W=\frac 14\sum_{\al,\beta>0}k_\alpha k_\beta B(\alpha,\beta) s_\alpha s_\beta$. It is easy to see that this can be rewritten as (see also \cite{BCT}):
\begin{equation}
\Omega_W=\frac 14\sum_{(\al,\beta)\in (\Phi^+)^2,\ s_\alpha(\beta)<0}k_\alpha k_\beta B(\alpha,\beta) s_\alpha s_\beta.
\end{equation}

Suppose $W=S_n$ and without loss of generality, assume that $k_\alpha=1$ for all $\alpha$. Assume the bilinear form is such that $B(\alpha,\alpha)=2$.

\begin{lm}\label{l:Sn-1} When $W=S_n$,  we have:

  \begin{enumerate}
  \item[(a)] $\Omega_{S_n}=\frac 14 (n(n-1)+e_{(123)})$, where $e_{(123)}$ is the sum of $3$-cycles in $\mathbb C[S_n]$.
  \item[(b)] $\Omega_{S_n}=\frac 14 (\frac {n(n-1)}2+\sum_{i=1}^n T_i^2)$, where $T_i=(1,i)+(2,i)+\dots+ (i-1,i)$, $1\le i\le n$, are the Jucys-Murphy elements.
    \item[(c)] If $\lambda$ is a partition of $n$ and $\sigma_\lambda$ is the corresponding irreducible $S_n$-representation, then \[\sigma_\lambda(\Omega_{S_n})=\frac 14(\frac {n(n-1)}2+\Sigma_2(\lambda)),\] where $\Sigma_k(\lambda)$ is the sum of $k$-powers of contents in $\lambda$ viewed as a left-justified decreasing Young diagram.
    \end{enumerate}
\end{lm}

\begin{proof}
Parts (a) and (b) are immediate by direct calculation. Part (c) follows from (b) via the known properties of Jucys-Murphy elements, or equivalently, by Frobenius character formula applied in the case of $3$-cycles. 
\end{proof}

Now suppose $W=W_n$ is the Weyl group of type $B_n$. Let $k_l$ be the parameter on the long roots $\epsilon_i\pm \epsilon_j$ and $k_s$ the parameter on the short roots $\epsilon_i$. Let $B$ be the standard bilinear form $B(\epsilon_i,\epsilon_j)=\delta_{ij}$. The irreducible $W_n$-representations are parameterised by pairs of partitions $(\lambda,\mu)$, where $\lambda$ is a partition of $a$ and $\mu$ is a partition of $b$, such that $a+b=n$. If we denote the corresponding character by $\chi_{(\lambda,\mu)}$, then we also have
\[
\chi_{(\lambda,\mu)}=\operatorname{Ind}_{W_a\times W_b}^{W_n}(\chi_{(\lambda,0)}\otimes \chi_{(0,\mu)}),
\]
where $\chi_{(\lambda,0)}$ is the character of the representation $\sigma_\lambda$ of $S_n$ inflated so that the short reflections act by the identity, and $\chi_{(0,\mu)}$ is the character of the representation $\sigma_\mu$ of $S_n$ inflated so that the short reflections act by the negative of the identity.

\begin{lm}\label{l:Bn-1}
  In the case of the Weyl group of type $B_n$:
  \begin{enumerate}
\item[(a)]  \[\Omega_{W_n}=\frac 14 (2 n(n-1) k_l^2+n k_s^2)+\frac 14 k_l^2 e_{A_2}+\frac 12 k_l k_s e_{B_2},   \]
  where $e_{A_2}$ and $e_{B_2}$ are the sums of $W_n$-conjugates of the Coxeter elements of type $A_2$ and $B_2$, respectively.
\item[(b)] The scalar by which $\Omega_{W_n}$ acts in the irreducible representation labeled by $(\lambda,\mu)$ is
  \[\chi_{(\lambda,\mu)}(\Omega_{W_n})=\frac 14 (2 n(n-1) k_l^2+n k_s^2)+2 k_l^2 (\sigma_\lambda(e_{(123)})+\sigma_\mu(e_{(123)})))+k_l k_s (\sigma_\lambda(e_{(12)})-\sigma_\mu(e_{(12)})).
  \]
  Moreover, $\sigma_\lambda(e_{(123)})=-\frac {a(a-1)}2+\Sigma_2(\lambda)$, $\sigma_\lambda(e_{(12)})=\Sigma_1(\lambda)$ and similarly for $\mu$.
  \end{enumerate}
\end{lm}

\begin{proof}
  The pairs of distinct positive roots $(\alpha,\beta)$ with $s_\alpha(\beta)<0$ that contribute are:
  \begin{itemize}
  \item $(\epsilon_i-\epsilon_j,\epsilon_i-\epsilon_k)$, $(\epsilon_i-\epsilon_j,\epsilon_k-\epsilon_j)$, $i<k<j$;
  \item $(\epsilon_i+\epsilon_j,\epsilon_i-\epsilon_k)$, $i<j$, $i<k$;
  \item $(\epsilon_i+\epsilon_j,\epsilon_i+\epsilon_k)$, $i<j<k$;
  \item $(\epsilon_i+\epsilon_j,\epsilon_i)$ and $(\epsilon_i+\epsilon_j,\epsilon_j)$, $i<j$;
  \item $(\epsilon_i,\epsilon_i+\epsilon_j)$, $i<j$;
    \item  $(\epsilon_i,\epsilon_i-\epsilon_j)$, $i<j$.
    \end{itemize}
   From this, we see that the only pairs of distinct roots $(\alpha,\beta)$ that contribute are the ones that form subroot systems of type $A_2$ or $B_2$. So $\Omega_{W_n}=a+b e_{A_2}+ c e_{B_2}$ for some constants $a,b,c$. We compute $a=\sum_{\alpha>0} k_\alpha^2$. For $b$, we only need to know how many times a representative of $e_{A_2}$ appears and this is a calculation we've already done for type $A$. For $c$, it is the same calculation for a representative of type $B_2$, and here we see that each such representative can be obtain from $(\epsilon_i,\epsilon_i-\epsilon_j)$ but also $(\epsilon_i+\epsilon_j,\epsilon_i)$ since $s_{\epsilon_i}s_{\epsilon_i-\epsilon_j}=s_{\epsilon_i+\epsilon_j} s_{\epsilon_i}$. Claim (a) follows.

\smallskip

    For (b), we use the character formula for induced representations:
    \[\chi_{(\lambda,\mu)}(w)=\frac 1{|W_a| |W_b|}\sum_{s^{-1}ws^\in W_a\times W_b}(\chi_{(\lambda,0)}(s^{-1} w s)\chi_{(0,\mu)}(s^{-1} w s)).
    \]
    Let $w=(123)$. The number of $s$ such that $s^{-1}w s\in W_a$ is $2^n a(a-1)(a-2)\cdot (n-3)!$. It follows that:
    \[\chi_{(\lambda,\mu)}((123))=\frac 1{2^n a! b!} ( 2^n a (a-1) (a-2) (n-3)! \sigma_\lambda((123)) \sigma_\mu(1)+2^n b (b-1) (b-2) (n-3)! \sigma_\lambda(1) \sigma_\mu((123)).
    \]
    Noting that $\chi_{(\lambda,\mu)}=\frac{n!}{a! b!}  \sigma_\lambda(1)\sigma_\mu(1)$ and that the size of the conjugacy class of $(123)$ in $W_n$ is $\frac 83 n(n-1)(n-2)$, it follows immediately that
    \[\chi_{(\lambda,\mu)}(e_{(123)})=8 (\sigma_\lambda(e_{(123)})+\sigma_\mu(e_{(123)})).
      \]
      Completely similarly, we deduce that
      \[\chi_{(\lambda,\mu)}(e_{B_2})=2 (\sigma_\lambda(e_{(12)})-\sigma_\mu(e_{(12)})),
      \]
      using that $\chi_{(\lambda,0)}(w(B_2))=\sigma_\lambda({(12)})$ and  $\chi_{(0,\mu)}(w(B_2))=-\sigma_\mu({(12)})$, if $w(B_2)$ is a representative of the conjugacy class of type $B_2$ in $B_n$. 
  \end{proof}

\subsection{The element $\Omega_W'$}\label{s:4.4}

We now look closer at the element $\Omega_W'=\tfrac {1}{16}\sum_{\al,\beta>0}k_\al k_\beta B(\al^\vee,\beta^\vee)^{-1} [s_\al, s_\beta]^2\in\bC W^W.$ Firstly, remark that, just as for $\Omega_W$, the element $\Omega_W'$ can be rewritten as
\begin{equation}
\Omega_W'=\frac 1{16} \sum_{(\al,\beta)\in(\Phi^+)^2,~s_\alpha(\beta)<0} k_\al k_\beta B(\al^\vee,\beta^\vee)^{-1} [s_\al, s_\beta]^2.
\end{equation}
This is because the terms corresponding to the pairs $(\alpha,\beta)$ and $(\alpha,\gamma)$, if $\gamma=s_\alpha(\beta)>0$, cancel out.

The image of $\Omega_W'$ in the Weyl algebra satisfy the following properties.
\begin{pp}\label{p:O-inv}
The element $\nu'(\Omega'_W)$ is an $O(V,B)$-invariant element of $\CW$.
\end{pp}

\begin{proof}
It suffices to show that under the faithful representation
$m:\C W(\C V,\omega)\rightarrow {\rm End}(S(V))$
the element $m(\Omega'_W)$ commutes with $m(\nu'(X))$, for all $X\in\mathfrak{so}(V,B)$.
Indeed, if that is the case, then $\Omega'_W$ will be an
$SO(V,B)$-invariant element of $\CW$ which also commutes with some 
reflections $s\in O(V,B)\setminus SO(V,B)$. This then implies that
$\Omega'_W$ is $O(V,B)$-invariant.

Now, to each pair of positive roots $(\alpha,\beta)\in \Phi^+\times\Phi^+$, let $X_{\alpha\beta}:=m(\nu'([s_\alpha,s_\beta]))$. Then, we claim that for every
$\xi = v_1\cdot v_2\cdot\ldots\cdot  v_m\in S^m(V)$ we have
\begin{equation}\label{e:diagsq}
m(\Omega_W')(\xi) = \sum_{j=1}^m \sum_{\al,\beta > 0}\frac{k_\alpha k_\beta}{16}\frac1{B(\alpha^\vee,\beta^\vee)} ( v_1\cdot\ldots\cdot X_{\alpha\beta}^2(v_j)\cdot\ldots\cdot  v_m) = \sum_{j=1}^m  v_1\cdot\ldots\cdot  \Omega'_W(v_j)\cdot\ldots\cdot  v_m.
\end{equation}
Indeed, the composition $X_{\alpha\beta}^2 = X_{\alpha\beta}\circ X_{\alpha\beta}$ acts
on $S(V)$ via
\begin{multline*}
X_{\alpha\beta}^2(\xi) = \sum_{j=1}^m v_1\cdot\ldots\cdot X^2_{\alpha\beta}(v_j)\cdot\ldots\cdot  v_m + \\
\sum_{j=1}^m\left(\sum_{i<j} v_1 \cdot\ldots\cdot  X_{\alpha\beta}(v_i) \cdot\ldots\cdot X_{\alpha\beta}(v_j) \cdot\ldots\cdot  v_m + 
\sum_{j<k} v_1\cdots X_{\alpha\beta}(v_j) \cdot\ldots\cdot X_{\alpha\beta}(v_k) \cdot\ldots\cdot  v_m \right).
\end{multline*}
However, on each copy of $S^2(V)$ occurring in between parenthesis, the element 
$X_{\alpha\beta}(v_i)X_{\alpha\beta}(v_j)$ corresponds to the action of 
$[s_\alpha, s_\beta]\in\bC W$ on $v_iv_j\in S^2(V)$. Hence, 
\[
\sum_{\al,\beta > 0} \frac{k_\alpha k_\beta}{16} \frac1{B(\alpha^\vee,\beta^\vee)} X_{\alpha\beta}(v_i)X_{\alpha\beta}(v_j) = 
\sum_{\al,\beta > 0} \frac{k_\alpha k_\beta}{16} \frac1{B(\alpha^\vee,\beta^\vee)}  [s_\alpha,s_\beta](v_iv_j) = 0,
\]
since $[s_\alpha,s_\beta]\in \bC W$ is anti-symmetric on $\alpha,\beta$, settling the claim. Thus, for all $X\in \mathfrak{so}(V,B)$ and $\xi = v_1v_2\cdots v_m\in S^mV$, using (\ref{e:diagsq}) we have
\[
m(X)m(\Omega_W')(\xi) =  m(X)\left(\sum_{j=1}^m  v_1 \cdot\ldots\cdot  \Omega'_W(v_j) \cdot\ldots\cdot  v_m\right) = m(\Omega_W')m(X)(\xi)
\]
since $X$ commutes with $\Omega_W'$ on $S^1(V) = V$.
\end{proof}

\begin{pp}\label{p:Casimirs}
The elements $\nu'(\Omega_W')$ and $\Omega(\mathfrak{sl}(2))$ satisfy the following linear relation
\[
\nu'(\Omega_W') = \left(\frac{(n-4)(\nu'(\Omega_W'))_0}{n-1}\right) + \left(\frac{-4(\nu'(\Omega_W'))_0}{n(n-1)}\right)\Omega(\mathfrak{sl}(2)).
\]
\end{pp}

\begin{proof}
By Proposition \ref{p:O-inv}, $\nu'(\Omega'_W)$ is  $O(V,B)$-invariant in $\CW$, and therefore, it must lie in $\nu(U(\mathfrak{sl}(2)))$ by the dual pair argument. In addition, $\nu'(\Omega'_W)$ commutes with $\mathfrak{sl}(2)$ itself as well, since it is built of elements of $\mathfrak{so}(V,B)$. Hence $\nu'(\Omega'_W)$ is in the centre of $\nu(U(\mathfrak{sl}(2)))$.  It follows thus that we can write
\begin{equation}\label{e:eq1}
\nu'(\Omega_W') = a + b\Omega(\mathfrak{sl}(2)).
\end{equation}
for some constants $a,b\in \bC$. Thus, taking the zero degree components we get, using Proposition \ref{p:sl2Cas}
\begin{equation}\label{e:eq2}
(\nu'(\Omega_W'))_0 = a -(\tfrac{3n}{4})b.
\end{equation}
Furthermore, acting with (\ref{e:eq1}) on $S^0(V)$, we get
\begin{equation}\label{e:eq3}
0 = a + \frac{n(n-4)}{4}b.
\end{equation}
Solving for $a$ and $b$ in (\ref{e:eq2}) and (\ref{e:eq3}) yields the claim.
\end{proof}

Therefore, to obtain a precise relation between $\nu'(\Omega'_W)$ and $\Omega(\mathfrak{sl}(2))$ of $\C W$, 
it suffices to compute the degree zero component of $\eta(\nu'(\Omega'_W))$.

\begin{pp}	\label{p:0-term}
We have
\[
\Big(\nu'(\Omega'_W)\Big)_0 = \tfrac{1}{2}B(\rho_k,\rho_k) - \tfrac{1}{8}\sum_{\alpha,\beta>0}k_\al k_\beta\frac{B(\alpha,\beta)^3}{B(\alpha,\alpha)B(\beta,\beta)},
\]
where $\rho_k = \tfrac{1}{2}\sum_{\alpha>0}k_\al \al \in V^*$.
\end{pp}

\begin{proof}
Let $\textup{ev}_0:S(V\oplus V^*)\to \bC$ be the evaluation at $0$ map. We have
\[
\Big(\nu'(\Omega'_W)\Big)_0 = \textup{ev}_0(\gamma(\nu'(\Omega'_W))(1)).
\]
We recall that for $x\in V\oplus V^*$ we have $\gamma(x) = \epsilon(x)+\tfrac{1}{2}i(x)$
so that $\gamma(x)(P)=x\cdot P+\tfrac{1}{2}i(x)(P)$. Moreover,
$i(x)(y)=\omega(x,y)$, for all $y\in V\oplus V^*$. Thus, if $v\in V,\lambda\in V^*$ and
$P\in \CW$ we note that
\begin{equation}\label{e:gtrick}
\gamma(v)(\gamma(\lambda)(P)) = v\cdot \gamma(\lambda)(P) + \tfrac{1}{2}i(v)(\lambda P)
+ \tfrac{1}{4}i(v)(i(\lambda)(P)).
\end{equation}
Now, for each pair of roots $\alpha,\beta$, we let $X_{\al\beta} = \nu'([s_\alpha,s_\beta])=B(\al^\vee,\beta^\vee)(\iota(\al)\beta-\iota(\beta)\al)$ and hence
\[
\gamma(X_{\al\beta})(1) = B(\al^\vee,\beta^\vee)(\gamma(\iota(\al))(\gamma(\beta)(1))-\gamma(\iota(\beta))(\gamma(\al)(1))) = B(\al^\vee,\beta^\vee)(\iota(\al)\cdot\beta-\iota(\beta)\cdot\al).
\]
Thus, using (\ref{e:gtrick}) and applying $\textup{ev}_0$ we obtain that
\begin{align*}
\Big(\nu'([s_\alpha,s_\beta]^2\Big)_0 &= B(\al^\vee,\beta^\vee)\textup{ev}_0(\gamma(X_{\al\beta})(\iota(\al)\cdot\beta-\iota(\beta)\cdot\al)\\
&= 
\tfrac{1}{4}B(\al^\vee,\beta^\vee)(i(\iota(\alpha))(i(\beta)(X_{\alpha\beta})) -
i(\iota(\beta))(i(\alpha)(X_{\alpha\beta}))\\
&=\tfrac{1}{4}B(\al^\vee,\beta^\vee)^2\Big(2\omega(\beta,\iota(\beta))\omega(\alpha,\iota(\alpha))-\omega(\beta,\iota(\al))^2-\omega(\alpha,\iota(\beta))^2\Big)\\
&=\tfrac{1}{2}B(\al^\vee,\beta^\vee)^2(B(\beta,\beta)B(\alpha,\alpha)-B(\alpha,\beta)^2)\\
&=2B(\al^\vee,\beta^\vee)B(\al,\beta)\left(1-\frac{B(\alpha,\beta)^2}{B(\al,\al)B(\beta,\beta)}\right).
\end{align*}
Hence, as $\Omega_W'=\tfrac {1}{16}\sum_{\al,\beta>0}k_\al k_\beta B(\al^\vee,\beta^\vee)^{-1} [s_\al, s_\beta]^2$ we get
\[
\Big(\nu'(\Omega'_W)\Big)_0 = \tfrac{1}{2}B(\rho_k,\rho_k) - \tfrac{1}{8}\sum_{\alpha,\beta>0}k_\al k_\beta\frac{B(\alpha,\beta)^3}{B(\alpha,\alpha)B(\beta,\beta)}
\]
as required.
\end{proof}

\begin{cor}\label{c:deg0vals}
For crystallographic root systems, following the normalization conventions of \cite{Bou}, we have:
\begin{align*}
\textup{Type } A_n&\qquad (\nu'(\Omega'_W))_0 =  \frac{k^2}{32}(n+1)n(n-1)\\
\textup{Type } B_n&\qquad (\nu'(\Omega'_W))_0 =  \frac{1}{8}k_l n(n-1)((n-2)k_l + k_s)\\
\textup{Type } C_n&\qquad (\nu'(\Omega'_W))_0 =  \frac{1}{8}k_s n(n-1)((n-2)k_s + 2k_l)\\
\textup{Type } D_n&\qquad (\nu'(\Omega'_W))_0 =  \frac{k^2}{8}n(n-1)(n-2)\\
\textup{Type } E_6&\qquad (\nu'(\Omega'_W))_0 =  \frac{45}{2} k^2\\
\textup{Type } E_7&\qquad (\nu'(\Omega'_W))_0 =  63 k^2\\
\textup{Type } E_8&\qquad (\nu'(\Omega'_W))_0 =  210 k^2\\
\textup{Type } F_4&\qquad (\nu'(\Omega'_W))_0 =  \frac{3}{2}(k_s+2k_l)(k_s+k_l)\\
\textup{Type } G_2&\qquad (\nu'(\Omega'_W))_0 =  \frac{3}{16}(k_s+3k_l)(k_s+k_l),
\end{align*}
where $k_s$ and $k_l$ are the parameters for the short and long roots, respectively.
\end{cor}

\begin{proof}
Suppose first that $\Phi$ is a simply laced root system. Assume that $k=1$. We claim that 
$\left(\nu'(\Omega'_W)\right)_0=\frac 38 \left(B(\rho,\rho)-\frac N 2\right)$, where $N$ is the number of positive roots. Indeed, notice that in this case $B(\alpha,\beta)^3=B(\alpha,\beta)$ for all positive roots $\alpha\neq \beta$. Then, from Proposition \ref{p:0-term} we get that 
\begin{align*}
\left(\nu'(\Omega'_W)\right)_0&=\frac 12 B(\rho,\rho)-\frac 18 \sum_{\alpha>0} B(\alpha,\alpha)-\frac 1{32} \sum_{\alpha\neq\beta>0} B(\alpha,\beta)\\
&=\frac 12 B(\rho,\rho)-\frac N 4+\frac 1{32} \sum_{\alpha>0}B(\alpha,\alpha)-\frac 1{32} \sum_{\alpha,\beta>0} B(\alpha,\beta)\\
&=\frac 12 B(\rho,\rho)-\frac 3{16} N-\frac 18 B(\rho,\rho)=\frac 38 B(\rho,\rho)-\frac 3{16} N.
\end{align*}
For type $A_n$, we substitute $B(\rho,\rho)=\tfrac{1}{12}n(n+1)(n+2)$ and $N=n(n+1)/2$. For  
 type $D_n$, we substitute $B(\rho,\rho)=\tfrac{1}{6}n(n-1)(2n-1)$ and $N=n(n-1)$. For  
 types $E_6, E_7$ and $E_8$, we substitute $B(\rho,\rho)=78, 399/2$ and $620$; $N=36, 63$ and $120$, respectively.
 
For type $B_n$, note that we also have $B(\alpha,\beta)^3=B(\alpha,\beta)$ whenever $\alpha\neq \beta$
are positive roots. Proceeding similarly to the symply-laced case, and using $\alpha^\vee = 2B(\alpha,\alpha)^{-1}\iota(\alpha)$, we get
\begin{align*}
\left(\nu'(\Omega'_W)\right)_0&=\frac 12 B(\rho_k,\rho_k)-\frac 18 \sum_{\alpha>0} k_\alpha^2B(\alpha,\alpha)-\frac 1{32} \sum_{\alpha\neq\beta>0} k_\alpha k_\beta B(\alpha^\vee,\beta^\vee)\\
&=\frac 12 B(\rho_k,\rho_k)- \frac 18 B(\rho^\vee_k,\rho^\vee_k)- \frac{3N_l k_l^2}{16},
\end{align*}
where $N_l$ is the number of long positive roots. Substituting $B(\rho_k,\rho_k) = \tfrac{1}{12}(3nk_s^2 + 6n(n-1)k_sk_l + 2n(n-1)(2n-1)k_l^2)$, $B(\rho^\vee_k,\rho^\vee_k) = \tfrac{1}{6}(6nk_s^2 + 6n(n-1)k_sk_l  + n(n-1)(2n-1)k_l^2)$ and $N_l = n(n-1)$ yields the claim. Type $C_n$ is obtained from type $B_n$ by setting $k'_s = k_l$, $k'_l = 2k_s$, where $k'_s,k'_l$ are the parameters of $C_n$ and $k_s,k_l$ are the ones of type $B_n$. 

Types $F_4$ and $G_2$ are obtained by direct computation using the conventions in Planche VIII and Planche IX of \cite{Bou}.
\end{proof}

%
%
%

\begin{rem}
In the case where the root system is crystallographic, $n>1$ and  $k_\alpha = 1$ for all positive roots, we obtain from Corollary \ref{c:deg0vals} and Proposition \ref{p:Casimirs}, that 
\[
(\nu'(\Omega'_W))_4 = -c(\Omega(\mathfrak{sl}(2)))_4
\]
for a rational constant $c>0$.
\end{rem}

Finally, we compute $\Omega'_W\in \mathbb C[W]$ in two examples. 

\begin{ex}
If $W=S_n$, then $\Omega_{S_n}'=\frac 18(e_{(123)}-|C_{(123)}|)$, and it acts in an irreducible $S_n$-representation $\sigma_\lambda$ by $\sigma_\lambda(\Omega_{S_n})=\frac 18(\Sigma_2(\lambda)-\Sigma_2((n))).$
\end{ex}

\begin{proof}
It is easy to see that $\Omega_{S_n}'=\frac 1{16} 2\sum_{i<k<j} ((ikj)-(ijk))^2=\frac 18(-2\sum_{i<k<j}1+e_{(123)})=\frac 18(e_{(123)}-|C_{(123)}|).$ The second claim follows from the action of $e_{(123)}$ as in Lemma \ref{l:Sn-1}.
\end{proof}

\begin{ex}
If $W=W_n$ (type $B_n$), then $\Omega_{W_n}'=\frac 1{16} k_l^2(e_{A_2}-|C_{A_2}|)+\frac 14 k_s k_l (e_{2\widetilde A_1}-|C_{2\widetilde A_1}|)$, where $2 \widetilde A_1$ denotes the conjugacy class of products of two commuting reflections in the short roots. 
\end{ex}

\begin{proof}
The proof is similar to that of Lemma \ref{l:Bn-1}(a). First we notice that $[s_\alpha,s_\beta]^2=s_\beta s_\gamma +s_\gamma s_\beta-2$, where $\gamma=-s_\alpha(\beta)>0$. Then from the list of pairs of roots that can contribute, we see that $\Omega_{W_n}'$ must be of the form $\Omega_{W_n}'=a+ b e_{A_2} + c e_{2\widetilde A_1}$. Since every commutator vanishes on the trivial representation, we see that in fact $\Omega_{W_n}'=b (e_{A_2}-|C_{A_2}|) + c (e_{2\widetilde A_1}-|C_{2\widetilde A_1}|)$. To determine $b$ and $c$, we proceed as for Lemma \ref{l:Bn-1}(a) and count the number of occurrences of a representative of type $A_2$ and of type $2\widetilde A_1$. This is a direct calculation for the root system of type $B_2$.

\end{proof}

\subsection{Unitary structures} As before $\C P$ can be endowed with a positive definite hermitian  form $\langle~,~\rangle_{\C P}$. For the Hecke algebra $\bH$, there exist two natural star operations $*$ and $\bullet$ defined on generators as follows:
\begin{equation}
\begin{aligned}
w^*=w^{-1},\quad \wti v^*=-\wti v,\\
w^\bullet=w^{-1},\quad \wti v^\bullet=\wti v,
\end{aligned}
\end{equation}
for all $w\in W$, $v\in V_0$.

Define two star operations on $\bH\otimes\C W$ as well:
\begin{equation}
(h\otimes \nu)^*=h^*\otimes \nu^*,\quad (h\otimes \nu)^\bullet=h^\bullet\otimes\nu^*,
\end{equation}
for $h\in\bH$ and $\nu\in\C W$.

\begin{lm}
In $\bH\otimes\C W$, $(D^\pm)^*=-D^{\mp}$ and $(D^\pm)^\bullet=D^{\mp}.$
\end{lm}

\begin{proof}
Straightforward.
\end{proof}

Let us assume that $M$ has a nondegenerate $*$-invariant hermitian form $\langle~,~\rangle_M$. Define the product form \[\langle~,~\rangle_{M\otimes\C P}=\langle~,~\rangle_M\langle~,~\rangle_{\C P}\] on $M\otimes\C P$, so that this becomes a $*$-invariant hermitian form on the $\bH\otimes\C W$-module $M\otimes\C P$. 

\begin{lm}\label{l:diff}
With the notation as above, if $x\in M(\sigma)\otimes \C P$, then
\[ \langle \C D^+ x, \C D^+ x\rangle_{M\otimes\C P}-\langle \C D^- x, \C D^- x\rangle_{M\otimes\C P}=(-\pi(\Omega)+\sigma(\Omega_W))\langle x,x\rangle_{M\otimes \C P}+\langle \C E_{\sigma,\C P}x,x\rangle_{M\otimes \C P}.
\]
\end{lm}

\begin{proof}
This follows immediately from (\ref{e:iso}) using the adjointness property $(\C D^{\pm})^*=-\C D^{\mp}.$
\end{proof}

As a particular example, notice that when $x\in M(\sigma)\otimes S^0(V^*)=M(\sigma)\otimes 1$, then $\C D^- x=0=\C E_{\sigma,\C P}x$, hence we have:

\begin{pp}\label{p:Casimir}
For all $x\in M(\sigma)\otimes 1$:
\[\langle \C D^+ x, \C D^+ x\rangle_{M\otimes\C P}=(\sigma(\Omega_W)-\pi(\Omega))\langle x,x\rangle_{M\otimes \C P}.\]
Moreover, if $M$ is $*$-unitary then $\pi(\Omega)=B(\chi_M,\chi_M)\le \sigma(\Omega_W)$, for all irreducible $W$-representations $\sigma$ such that $M(\sigma)\neq 0$.

In particular, if $M$ is $W$-spherical, i.e., $M(\triv)\neq 0$, and $*$-unitary, then $B(\chi_M,\chi_M)\le B(\rho_k,\rho_k)$.
\end{pp}

\begin{proof}
The inequality follows by taking $x\neq 0$ and using that $\langle \C D^+ x, \C D^+ x\rangle_{M\otimes\C P}\ge 0$, since $M\otimes\C P$ is a $*$-unitary $\bH\otimes\C W$-module.
\end{proof}

\begin{rem}
This is of course the analogue of the Casimir inequality for Lie algebras (and it has already been known in this setting by \cite{BCT}).\end{rem}

Now suppose more generally that $x\in M(\sigma)\otimes\C P$ is a simple tensor:  $x=m\otimes p$, $m\in M$, $p\in \CP$. Then:
\begin{equation}
\begin{aligned}
\langle \C E_{\sigma,\C P}x,x\rangle_{M\otimes \C P}&=\frac 14\sum_{\{\al,\beta\}\in\Phi_2^+}k_\al k_\beta \langle \sigma([s_\al, s_\beta]) m\otimes (\iota(\al)\partial_\beta-\iota(\beta)\partial_\al)p,m\otimes p\rangle _{M\otimes \C P}\\
&=\frac 14\sum_{\{\al,\beta\}\in\Phi_2^+}k_\al k_\beta\langle \sigma([s_\al,s_\beta])m,m\rangle_M \langle (\iota(\al)\partial_\beta-\iota(\beta)\partial_\al)p,p\rangle_{\C P}.
\end{aligned}
\end{equation} 
But $\langle (\iota(\al)\partial_\beta-\iota(\beta)\partial_\alpha)p,p\rangle_{\C P}=\langle\partial_\beta(p),\partial_\alpha (p)\rangle_{\C P}-\langle\partial_\al(p),\partial_{\beta}(p)\rangle_{\C P}=0$, when $p\in S(V_0)$ (i.e., $p$ is a real polynomial). Notice that we used in this calculation that multiplication by $\iota(\al)$ is adjoint to $\partial_\al$ in the inner product on $\C P$. In conclusion,
\begin{equation}\label{e:zero}
\langle \C E_{\sigma,\C P}x,x\rangle_{M\otimes \C P}=0,\text{ for all } x=m\otimes p,\ p\in S(V_0^*).
\end{equation}

\begin{pp} Suppose $M$ is $*$-hermitian as above and $\sigma$ is an irreducible $W$-representation such that $M(\sigma)\neq 0.$ If $x=m\otimes p\in M(\sigma)\otimes S(V_0)$ is a simple tensor, then 
\[
\langle \C D^+ x, \C D^+ x\rangle_{M\otimes\C P}-\langle \C D^- x, \C D^- x\rangle_{M\otimes\C P}=(-\pi(\Omega)+\sigma(\Omega_W))\langle x,x\rangle_{M\otimes \C P}.
\]
\end{pp}

\begin{proof} The formula follows at once from Lemma \ref{l:diff} and (\ref{e:zero}).

\end{proof}

\end{document}